\documentclass[preprint,10pt]{elsarticle}

 \usepackage{epsfig}
\usepackage[all]{xy}
\usepackage{amssymb}
\usepackage{amsthm}

\newtheorem{theorem}{Theorem}[section]
\newtheorem{lemma}[theorem]{Lemma}
\newtheorem{proposition}[theorem]{Proposition}
\newtheorem{corollary}[theorem]{Corollary}
\theoremstyle{remark}
\newtheorem{remark}[theorem]{Remark}
\newtheorem{example}[theorem]{Example}

\journal{Journal of Algebra}

\begin{document}

\begin{frontmatter}



\title{Weak crossed-product orders over valuation rings}


\author{John S. Kauta}

\address{School of Computing, Information \& Mathematical Sciences\\ The University of the South Pacific\\
 Suva\\ FIJI ISLANDS.}

\begin{abstract}

Let $F$ be a field, let $V$ be a valuation ring of $F$ of arbitrary
Krull dimension (rank), let $K$ be a finite Galois extension of $F$
with group $G$, and let $S$ be the integral closure of $V$ in $K$.
Let $f:G\times G\mapsto K\setminus \{0\}$ be a normalized
two-cocycle such that $f(G\times G)\subseteq S\setminus \{0\}$, but
we do not require that $f$ should take values in the group of
multiplicative units of $S$. One can construct a crossed-product
$V$-order $A_f=\sum_{\sigma\in G}Sx_{\sigma}$ with multiplication
given by $x_{\sigma}sx_{\tau}=\sigma(s)f(\sigma,\tau)x_{\sigma\tau}$
for $s\in S$, $\sigma,\tau\in G$. We characterize semihereditary and
Dubrovin crossed-product orders, under mild valuation-theoretic assumptions
placed on the nature of the extension $K/F$.
\end{abstract}


\end{frontmatter}

% \linenumbers

\let\thefootnote\relax\footnote{J. Algebra 402 (2014) 319-350}
\section{Introduction}
\label{sec:intro}
This work is about orders in a crossed-product $F$-algebra $(K/F,G,f)$ which are integral over their center $V$,
a valuation ring of the field $F$,
and generalizes the results obtained in \cite{H,K5,K6, CW2}. Whereas in \cite{H,K5,K6} the valuation ring $V$
is unramified in $K$, and in \cite{H,K5,CW2} $V$ is a discrete valuation ring (DVR),
we seldomly make such assumptions in this paper. Further, our theory complements and enhances the classical theory of crossed-product orders over valuation rings, such as can be found in \cite{Hd, W}, in two significant ways: firstly, we do not require that the valuation ring $V$ should be a DVR, as mentioned above; secondly, we do not require that the values of the two-cocycles associated with our crossed-product orders should be multiplicative units in the integral closure of $V$ in $K$. We shall give a precise description of the crossed-product orders we will be dealing with shortly.

If $R$ is a ring, then $J(R)$ will
denote its Jacobson radical, $Z(R)$ its center, $U(R)$ its group of multiplicative
units, and $R^{\#}$ the subset of all the non-zero elements. The
residue ring $R/J(R)$ will be denoted by $\overline{R}$. Given the
ring $R$, it is called \textit{primary} if $J(R)$ is a maximal ideal
of $R$. It is called \textit{hereditary} if one-sided ideals are
projective $R$-modules. It is called \textit{semihereditary}
(respectively \textit{B\'{e}zout}) if finitely generated
one-sided ideals are projective $R$-modules (respectively are principal).
Let $V$ be a valuation ring of a field $F$. If $Q$ is
a finite-dimensional central simple $F$-algebra, then a subring $R$
of $Q$ is called an order in $Q$ if $RF=Q$. If in addition
$V\subseteq R$ and $R$ is integral over $V$, then $R$ is called a
$V$-order. If a $V$-order $R$ is maximal among the $V$-orders of $Q$
with respect to inclusion, then $R$ is called a \textit{maximal $V$-order}
(or just a \textit{maximal order} if the context is clear). A $V$-order $R$
of $Q$ is called an \textit{extremal} $V$-order (or simply \textit{extremal}
when the context is clear) if for every $V$-order $B$ in $Q$ with
$B\supseteq R$ and $J(B)\supseteq J(R)$, we have $B=R$. If $R$ is an
order in $Q$ then, since it is a PI-ring, it is called a \textit{Dubrovin valuation ring} of $Q$
(or a \textit{valuation ring} of $Q$ in short) if
it is semihereditary and primary (see \cite{D1}).

In this paper, $V$ will denote a commutative valuation domain of \textit{arbitrary}
Krull dimension (rank). Let $F$ be its field of quotients, let $K/F$
be a finite Galois extension with group $G$, and let $S$ be the
integral closure of $V$ in $K$. If $f\in Z^2(G,U(K))$ is a
normalized two-cocycle such that $f(G\times G)\subseteq S^{\#}$,
then one can construct a ``crossed-product'' $V$-algebra
$$A_f=\sum_{\sigma\in G} Sx_{\sigma},$$ with the usual rules of multiplication:
$x_{\sigma}sx_{\tau}=\sigma(s)f(\sigma,\tau)x_{\sigma\tau}$
for $s\in S$, $\sigma,\tau\in G$. Then $A_f$ is
associative, with identity $1=x_1$, and center $V=Vx_1$. Further,
$A_f$ is a $V$-order in the crossed-product $F$-algebra
$\Sigma_f=\sum_{\sigma\in G} Kx_{\sigma}=(K/F,G,f)$.

Two such cocycles $f$ and $g$ are said to be cohomologous over $S$
(respectively cohomologous over $K$), denoted by $f\! \sim_S\! g$
(respectively $f\! \sim_K\! g$), if there are elements
$\{c_{\sigma}\mid \sigma\in G\}\subseteq U(S)$ (respectively
$\{c_{\sigma}\mid \sigma\in G\}\subseteq K^{\#}$) such that
$g(\sigma,\tau)=c_{\sigma}\sigma{(c_{\tau})}c_{\sigma\tau}^{-1}f(\sigma,\tau)$
for all $\sigma,\tau \in G$. Following \cite{H}, let $H=\{\sigma\in
G\mid f(\sigma,\sigma^{-1})\in U(S)\}=\{\sigma\in G\mid x_{\sigma}\in U(A_f)\}$. Then $H$ is a subgroup of
$G$. On $G/H$, the left coset space of $G$ by $H$, one can define a
partial ordering by the rule $\sigma H \leq \tau H
\;\;\textrm{if}\;\; f(\sigma,\sigma^{-1}\tau)\in U(S).$ Then ``$\leq
$'' is well-defined, and depends only on the cohomology class of $f$
over $S$. Further, $H$ is the unique least element. We call this
partial ordering on $G/H$ \textit{the graph of f} and in this paper we will denote it by $\textrm{Gr}(f)$.

Such a setup was first formulated by Haile in \cite{H}, with the
assumption that $V$ is a DVR unramified in $K$, wherein, among other
things, conditions equivalent to such orders being maximal orders
were considered. Since we do not require that $H=G$, $A_f$ is a so-called ``weak" crossed-product order,
but we shall abuse the terminology and refer to it simply as a crossed-product order.
These are the crossed-product orders we will investigate in this paper, \textit{without} the assumption that $V$ is a DVR unramified in $K$.

 It is worth making the following two observations. First,
let $g\in Z^2(G,U(K))$ be an \textit{arbitrary} normalized two-cocycle
with values not necessarily in $S^{\#}$ and let $(K/F,G,g)=\sum_{\sigma\in G}Ku_{\sigma}$
be a crossed-product $F$-algebra with respect to $g$. We can choose a $t\in V^{\#}$
such that, for every $\sigma,\tau\in G$, we have $tg(\sigma,\tau)\in S^{\#}$.
Now let $x_1=u_1=1$ and, for $\sigma\not= 1$, let $x_{\sigma}=tu_{\sigma}$.
Then $x_{\sigma}x_{\tau}=f(\sigma,\tau)x_{\sigma\tau}$ where $f$ is a normalized
two-cocycle satisfying $f(G\times G)\subseteq S^{\#}$, and we have
$(K/F,G,g)=\sum_{\sigma\in G}Kx_{\sigma}\supseteq A_f$.
Thus the crossed-product orders under consideration in this paper occur
quite naturally as $V$-orders \textit{in every crossed-product $F$-algebra.}
Second, let us assume for the moment that $V$ is not necessarily a valuation ring,
but just a commutative domain whose quotient field is $F$. Then $A=A_f$ would still
be an order in $\Sigma_f$ integral over $V$. Therefore if $A$ is semihereditary
then by \cite[Theorem 1.3]{MMUZ} $V$ would be a Pr\"{u}fer domain, that is,
a (commutative) semihereditary domain. But by \cite[Lemma 2.2]{M2} $A$ is
semihereditary if and only if $A_m$ is semihereditary for
each maximal ideal $m$ of $V$, where $A_m$ denotes the central localization
at $m$. Since our interest in this paper is in crossed-product orders that
are at least semihereditary, it is no restriction to assume that $V$ is a
valuation ring of $F$.

Our valuation theoretic terminology will be consistent with that of \cite{E}. If $U$ is a valuation ring of a field $L$, we will denote its value group by $\Gamma_U$. Let $M$ be a maximal ideal of $S$. Then $e(S_M\!\!\mid\!\! F):= [\Gamma_{S_M}:\Gamma_V]$ is called the ramification index of $S_M$ over $F$. The extension $K/F$ is said to be \textit{defectless} if $[K:F]=e(S_M\!\!\mid\!\! F)[\overline{S}_M:\overline{V}]r$, where $r$ is the number of maximal ideals of $S$.
The extension $K/F$ is said to be \textit{tamely ramified} if $\textrm{Char(}\overline{V}\textrm{)}$ does not divide $e(S_M\!\!\mid\!\! F)$ and $\overline{S}_M$ is separable over $\overline{V}$. The extension is called \textit{unramified} if $e(S_M\!\!\mid\!\! F)=1$ and $\overline{S}_M$ is separable over $\overline{V}$.
We let $G^Z(S_M\!\!\mid\!\! F)=\{\sigma\in G\mid \sigma(S_M)\subseteq S_M\}$, the decomposition group of $S_M$ over $F$, and we denote its fixed field, called the decomposition field, by $K^Z(S_M\!\!\mid\!\! F)$. We let
$G^T(S_M\!\!\mid\!\! F)=\{\sigma\in G\mid \sigma(x)-x\in J(S_M)\forall x\in S_M\}$, the inertial group of $S_M$ over $F$, and we denote its fixed field, called the inertia field, by $K^T(S_M\!\!\mid\!\! F)$. Since $K$, $F$, $G$, $S$, $V$, and $M$ will be fixed in this paper, we shall merely refer to $G^Z(S_M\!\!\mid\!\! F)$ as $G^Z$, $G^T(S_M\!\!\mid\!\! F)$ as $G^T$, and so on, if there is no danger of confusion.

Let $\overline{p}=\max(1,\textrm{Char(}\overline{V}\textrm{)})$, the \textit{characteristic exponent} of $\overline{V}$.
Since $K/F$ is a Galois extension, by \cite[Lemma 1]{K3}, $K/F$ is tamely ramified and defectless if and only if $\gcd(\overline{p},\mid\!\! G^T\!\!\mid)=1$, and $K/F$ is unramified and defectless if and only if $\mid\!\! G^T\!\!\mid=1$.

Let $L$ be an intermediate field of $F$ and $K$, let $G_L$ be the Galois
group of $K$ over $L$, let $U$ be a valuation ring of $L$ lying over
$V$, and let $T$ be the integral closure of $U$ in $K$. Then one can
obtain a two-cocycle $f_{L,U}:G_L\times G_L\mapsto T^{\#}$ from $f$
by restricting $f$ to $G_L\times G_L$, and embedding $S^{\#}$ in
$T^{\#}$. As before, $A_{f_{L,U}}=\sum_{\sigma\in G_L}Tx_{\sigma}$
is a $U$-order in $\Sigma_{f_{L,U}}=\sum_{\sigma\in
G_L}Kx_{\sigma}=(K/L,G_L,f_{L,U})$. If $M$ is a maximal ideal of $S$ and $U=K^Z\cap S_M$, then we will denote
$f_{K^Z,U}$ by $f_M$, $A_{f_{K^Z,U}}$ by $A_{f_M}$, and $\Sigma_{f_{K^Z,U}}$
by $\Sigma_{f_M}$, following \cite{H}. Further, we let $H_M=\{\sigma\in G^Z\mid
f_M(\sigma,\sigma^{-1})\not\in M\}$, a subgroup of $G^Z$.
On the other hand, if $U=K^T\cap S_M$, then we will denote
$f_{K^T,U}$ by $f_{(M)}$, $A_{f_{K^T,U}}$ by $A_{f_{(M)}}$, and $\Sigma_{f_{K^T,U}}$
by $\Sigma_{f_{(M)}}$.

In Section~\ref{sec:General-Theory}, we describe the graded Jacobson radical of $A_f$ explicitly based on an argument of Haile \cite{H}, and characterize semihereditary and Dubrovin crossed-product orders when $J(V)$ is not a principal ideal of $V$. We also introduce the notion of a graph of $[f]_M$, denoted by $\textrm{Gr}(f^M)$, where $M$ is a maximal ideal of $S$, and give a sufficient condition for it to be a chain. The graph $\textrm{Gr}(f^M)$ seems to be a useful alternative to the graph $\textrm{Gr}(f)$ when $V$ decomposes in $K$. The two graphs coincide when $V$ is indecomposed in $K$.

In Section~\ref{sec:Tame-Case}, we assume for the most part that the extension $K/F$ is tamely ramified and defectless. With this assumption, we give an explicit description of the Jacobson radical of $A_f$, generalizing \cite[Proposition 3.1(b)]{H} and \cite[Theorem 4.3]{CW2}. When $V$ is a DVR unramified and indecomposed in $K$, Haile \cite[Theorem 2.3]{H} showed that $A_f$ is a hereditary order (actually a maximal order in this case since it is primary) if and only if $H$ is a normal subgroup of $G$ with cyclic quotient and, if $H\not=G$, then the graph $\textrm{Gr}(f)$ is the chain $$H\leq\sigma H\leq\sigma^2H\leq\cdots\leq\sigma^{\mid G/H\mid-1}H$$ for some $\sigma\in G$ satisfying: 1) $G/H=\langle\sigma H\rangle$, 2) $f(\sigma,\sigma^{-1})\in J(S)\setminus J(S)^2$,  and 3) $J(A_f)=x_{\sigma}A_f$.
Wilson \cite[Theorem 2.15]{CW2} extended this result to the case when $V$ is a DVR tamely ramified and indecomposed in $K$. He also showed that, under this assumption, $G^T\subseteq H$ whenever $A_f$ is hereditary. We extend these results to an arbitrary valuation ring $V$ when $J(V)$ is a principal ideal of $V$. We provide various other characterizations of semihereditary crossed-product orders in this section when $V$ is not necessarily indecomposed in $K$, generalizing results in \cite{K5,K6}. We end the section by generalizing two classical results, one by Auslander and Rim and the other by Harada, and we also complete the characterization of the so-called ``hereditary global crossed product orders'' introduced in \cite{Ke}.

In Section~\ref{sec:Dubrovin-Case}, we characterize Dubrovin crossed-product orders under the assumption that $K/F$ is tamely ramified and defectless. Assuming $V$ is a DVR unramified in $K$, it was shown in \cite{H} that $A_f$ is primary if and only if given a maximal ideal $M$ of $S$, there exists a set of right coset representatives $\sigma_1,\sigma_2,\ldots,\sigma_r$ of $G^Z$ in $G$ that is ``nice'' in the sense that $f(\sigma_i,\sigma_i^{-1})\not\in M$ for all $i$. We generalize this result in this section. We also reduce the problem of determining when the crossed-product order $A_f$ is a valuation ring of $\Sigma_f$ to the case when $K/F$ is tame totally ramified and $H=G$.

In Section~\ref{sec:graphs}, we take a closer look at the connections between the various graphs encountered in this paper. If $M$ is a maximal ideal of $S$, we show that the ``nice'' set of right coset representatives of $G^Z$ in $G$ mentioned above exists precisely
when a certain graph isomorphism exists between $\textrm{Gr}(f_M)$ and $\textrm{Gr}(f^M)$. Just as in \cite{H}, we show that the existence of a ``nice'' set of right coset representatives with respect to one maximal ideal of $S$ implies the existence of a ``nice'' set of right coset representatives for all the maximal ideals of $S$, and this is \textit{not} premised on the fact that $V$ is tamely ramified or defectless in $K$.

As pointed out in \cite{K6}, the behavior of the crossed-product order $A_f$ when $J(V)$ is a principal ideal
of $V$ is very similar to its behavior when $V$ is a DVR. This is a recurrent theme throughout this paper. The recent PhD thesis
of Chris Wilson \cite{CW1}, part of which appears in \cite{CW2}, covers the case when $V$ is a DVR tamely ramified in $K$ very well. Some of the arguments in \cite{CW2} will be useful in this paper.

\section{Preliminaries and Generalities}
\label{sec:General-Theory}

Strictly speaking, $A_f$ is not a crossed-product of $G$ over $S$, but rather a $G$-graded ring. Therefore, following \cite{P}, if $G_1$ is a subgroup of
$G$, we will let $S(G_1)=\sum_{\sigma\in G_1}Sx_{\sigma}$. Note that $S(G)=A_f$ and, if $L$ is the fixed field of $G_1$, then $S(G_1)$ is an order in the crossed-product $L$-algebra $K(G_1)=(K/L,G_1,f_{\mid_{G_1\times G_1}})$.

A ring is said to be \textit{coherent} if finitely generated one-sided ideals are finitely presented.
If $R$ is a ring, $\textrm{wgld(}R\textrm{)}$ will denote its \textit{weak global dimension} (see \cite{MR} for
the definition).

Part (1) of the following lemma is known when $H=G$ (see the proof of \cite[Theorem 3.2]{Y}).

 \begin{lemma} \label{lem:Coherent-lemma}
 Let $G_1$ be a subgroup of $G$. Then:
 \begin{enumerate}
    \item $S(G_1)$ is a coherent ring. Consequently,
    \item $A_f$ is semihereditary if and only if $J(A_f)$ is a flat $A_f$-module, and
    \item $S(G_1)$ is semihereditary if and only if $wgld(S(G_1))\leq 1$.
    \end{enumerate}
 \end{lemma}
 \begin{proof}
 Let $I=\sum_{i=1}^na_iS(G_1)$ be a finitely generated right ideal of $S(G_1)$. Let $S(G_1)^{(n)}$ denote the right
 $S(G_1)$-module which is the external direct sum of $n$ copies of $S(G_1)$, and consider the $S(G_1)$-epimorphism $$\phi:S(G_1)^{(n)}\mapsto I:(x_1,x_2,\ldots,x_n)\mapsto\sum_ia_ix_i.$$ We will show that $\ker\phi$ is finitely generated over $S(G_1)$.
  Observe that $S(G_1)$ is a free $S$-module of finite rank, $I$ is a finitely generated $S$-submodule, and $S$ is a semihereditary ring. So $I$ must be a projective $S$-module \cite[Proposition 4.30]{Rt}, and $\ker\phi$ a finitely generated $S$-module by Schanuel's Lemma. Hence $\ker\phi$ is a finitely generated $S(G_1)$-module as claimed. In the same manner, one can show that every finitely generated left ideal of $S(G_1)$ is finitely presented, hence $S(G_1)$ is a coherent ring. Part (2) now follows from \cite[Theorem 2.7]{MMUZ}; as for part (3), see for instance \cite[Example 8.20(iii)]{Rt}.
 \end{proof}

 If $h\in H$ then, since $f^h(h^{-1},h\sigma)f(h,\sigma)=f(h,h^{-1})$ and $f^{\sigma}(h,h^{-1})=f(\sigma,h)f(\sigma h,h^{-1})$, we see that $f(h,\sigma),f(\sigma,h)\in U(S)$ for every $\sigma\in G$ (cf. \cite[Lemma 3.5]{H}).

 \begin{lemma} \label{lem:Subgroup-lemma}

 Let $G_1$ be a subgroup of $G$ such that $G_1\subseteq H$. Then:
 \begin{enumerate}
    \item $\textrm{wgld(}S(G_1)\textrm{)}\leq \textrm{wgld(}A_f\textrm{)}$.
    \item If $A_f$ is semihereditary, then $S(G_1)$ is semihereditary.
    \item Let $L$ be the fixed field of $G_1$, and let $U$ be a valuation ring of $L$ lying over $V$. If $A_f$ is semihereditary, then $A_{f_{L,U}}$ is semihereditary.
    \end{enumerate}
 \end{lemma}

 \begin{proof}
First, we will show that $A_f$ is a free right $S(G_1)$-module.
If $\sigma\in G_1$, let $S(\sigma G_1)=\sum_{\gamma\in G_1}Sx_{\sigma\gamma}$, a right $S(G_1)$-module. Since $G_1\subseteq H$, $f(\sigma,\gamma)\in U(S)$ for every $\gamma\in G_1$, hence $S(\sigma G_1)=x_{\sigma}S(G_1)$, a free $S(G_1)$-module with basis $\{x_{\sigma}\}$.  In particular, $A_f$ is a free right $S(G_1)$-module of finite rank: if $\sigma_1,\sigma_2,\ldots,\sigma_r$ are left coset representatives of $G_1$ in $G$, then $A_f=\sum_iS(\sigma_iG_1)$ (direct sum). Because $S(G_1)$ is an $S(G_1)$-bimodule direct summand of $A_f$, by \cite[Theorem 7.2.8(ii)]{MR} we conclude that $\textrm{wgld(}S(G_1)\textrm{)}\leq \textrm{wgld(}A_f\textrm{)}$ (cf. \cite[Lemma 2.1(ii)]{Y}). Therefore, if $A_f$ is semihereditary, then $S(G_1)$ is also semihereditary by Lemma~\ref{lem:Coherent-lemma}. Since $S(G_1)\subseteq A_{f_{L,U}}\subseteq \Sigma_{f_{L,U}}$, if $S(G_1)$ is semihereditary, so is
 $A_{f_{L,U}}$, by \cite[Lemma 4.10]{M}.
  \end{proof}

Recall that $J(V)$ is not a principal ideal of $V$ if and only if $J(V)^2=J(V)$. When this happens, $M^2=M$ for every maximal ideal $M$ of $S$.

For a $\sigma\in G$, here and elsewhere we will let $I_{\sigma}=\bigcap M$, where the intersection is taken over those maximal ideals $M$ of $S$ for which $f(\sigma,\sigma^{-1})\not\in M$.

\begin{lemma} \label{lem:the-ideal-lemma} (cf. \cite[Lemma 2.4]{K6}) Given a $\sigma\in G$, then:
\begin{enumerate}
    \item $I_{\sigma}=\{x\in S\mid xf(\sigma,\sigma^{-1})\in
    J(S)\}$.
    \item $\sigma^{-1}(I_{\sigma})=I_{\sigma^{-1}}$.
    \item If $f(\sigma,\sigma^{-1})\not\in M^2$ for every maximal
    ideal $M$ of $S$, then $I_{\sigma}f(\sigma,\sigma^{-1})=J(S)$.
    \item $\sum_{\sigma\in G} I_{\sigma}x_{\sigma}$ is an ideal of $A_f$.
    \item $I_{\sigma}\subseteq J(S)$ if and only if $\sigma\in H$. When this occurs,
    $I_{\sigma}= J(S)$.
\end{enumerate}
 \end{lemma}

 \begin{proof} Let $x\in S$. Clearly, if $x\in I_{\sigma}$ then $xf(\sigma,\sigma^{-1})\in
 J(S)$. On the other hand, if $x\not\in I_{\sigma}$ then there
 exists a maximal ideal $M$ of $S$ such that
 $x,f(\sigma,\sigma^{-1})\not\in M$, hence
 $xf(\sigma,\sigma^{-1})\not\in M$, and thus $xf(\sigma,\sigma^{-1})\not\in J(S)$.

 The second statement is proved in the same manner as
 \cite[Sublemma]{K5}. To see that the third statement holds,
we note that $I_{\sigma}f(\sigma,\sigma^{-1})\subseteq J(S)$. We
claim that $I_{\sigma}f(\sigma,\sigma^{-1})= J(S)$. If $J(V)$ is not a principal
ideal of $V$ then the assumption implies that $f(\sigma,\sigma^{-1})\in U(S)$, hence the conclusion holds.
To see that the conclusion also holds when $J(V)$ is a principal ideal of $V$,
let $M$ be a maximal ideal of $S$. If $f(\sigma,\sigma^{-1})\not\in
M$, then $(I_{\sigma}f(\sigma,\sigma^{-1}))S_M=J(S_M)=J(S)S_M$.
On the other hand, if $f(\sigma,\sigma^{-1})\in M$ then, since
$f(\sigma,\sigma^{-1})\not\in M^2$, we have $J(S_M)^2 \subsetneqq
I_{\sigma}f(\sigma,\sigma^{-1})S_M\subseteq J(S_M)$, hence
$I_{\sigma}f(\sigma,\sigma^{-1})S_M=J(S_M)=J(S)S_M$, and thus
$I_{\sigma}f(\sigma,\sigma^{-1})= J(S)$.

The argument in \cite[Proposition 3.1]{H} now shows that indeed,
$\sum_{\sigma\in G} I_{\sigma}x_{\sigma}$ is an ideal of $A_f$.

Finally, observe that $I_{\sigma}\subseteq J(S)$ implies that $I_{\sigma}\subseteq M$ for every maximal ideal $M$ of $S$, and hence $f(\sigma,\sigma^{-1})\in U(S)$. The converse also holds.
\end{proof}

Let $J_G(A_f)$ denote the \textit{graded Jacobson radical} of $A_f$ (see \cite{P} for the definition). By \cite[Theorem 4.12]{P}, we always have $J(V)A_f\subseteq J(S)A_f\subseteq J_G(A_f)\subseteq J(A_f)$. We are now going to give an explicit description of $J_G(A_f)$.

\begin{proposition} \label{prop:The-Graded-Radical-Structure}
 We always have $J_G(A_f)=\sum_{\sigma\in G}I_{\sigma}x_{\sigma}$.
 \end{proposition}

 \begin{proof}
 Let $I=\sum_{\sigma\in G}I_{\sigma}x_{\sigma}$, an ideal of $A_f$ by Lemma~\ref{lem:the-ideal-lemma}(4) . The argument by Haile in \cite[Proposition 3.1(b)]{H} is easily
 adaptable to this situation to establish that $I$ is the largest $G$-graded ideal of $A_f$
 contained in $J(A_f)$. Hence by \cite[Theorem 4.12]{P}, $I=J_G(A_f)$.
 \end{proof}

  \begin{theorem}  \label{thm:The-Azumaya-Theorem} The crossed-product order $A_f$ is Azumaya over $V$ if and only if $H=G$, and $K/F$ is unramified and defectless.
  \end{theorem}

  \begin{proof}
  Suppose $A_f$ is Azumaya over $V$. Then $J(A_f)=J(V)A_f$, hence $J_G(A_f)=J(V)A_f$ and $I_{\sigma}=J(V)S\subseteq J(S)$ for every $\sigma\in G$.  This shows that $H=G$ by Lemma~\ref{lem:the-ideal-lemma}(5). The theorem now follows from \cite[Theorem 3]{K3}.
  \end{proof}

  \begin{theorem} \label{thm:The-Main-Nonprincipal-Theorem}
  Suppose $J(V)$ is not a principal ideal of $V$.
  \begin{enumerate}
    \item The crossed-product order $A_f$ is an extremal $V$-order precisely when it is a maximal $V$-order.
    \item The following are equivalent:
    \begin{enumerate}
        \item $A_f$ is semihereditary.
        \item $A_f$ is a semihereditary maximal $V$-order.
        \item $J(A_f)=J(V)A_f$.
        \end{enumerate}
    When this happens, $H=G$ and $J(A_f)=J_G(A_f)$.
    \item Suppose $A_f$ is
        semihereditary. Then $A_{f_{L,U}}$ is a semihereditary order in
        $\Sigma_{f_{L,U}}$ for each intermediate field $L$ of $F$ and $K$,
and every valuation ring $U$ of $L$ lying over $V$.
    \item $A_f$ is a valuation ring of $\Sigma_f$ if and only if $J(V)A_f$ is a maximal ideal of $A_f$.
    \item Suppose $K/F$ is tamely ramified and defectless. Then $A_f$ is semihereditary if and only if $H=G$, if and only if $f(\sigma,\tau)\not\in M^2$ for every maximal ideal $M$ of $S$ and every $\sigma,\tau\in G$.
    \item If $\overline{V}$ is a perfect field, then $A_f$ is semihereditary if and only if $H=G$ and $K/F$ is
           tamely ramified and defectless.
    \end{enumerate}
  \end{theorem}

  \begin{proof} (1) Suppose $A_f$ is an extremal $V$-order. Let $B$ be a $V$-order in $\Sigma_f$ containing $A_f$. Since $A_f$ is extremal, $J(B)\subseteq A_f$, by \cite[Proposition 1.4]{K1}. Let $t=\sum_{\sigma\in G}a_{\sigma}x_{\sigma}\in B$, with $a_{\sigma}\in K$. Then $a_{\sigma}J(V)\subseteq S$, since $J(V)\subseteq J(B)\subseteq A_f$ and so $tJ(V)\subseteq A_f$. Since $J(V)^2=J(V)$, we see that $a_{\sigma}J(V)\subseteq J(V)S$. Thus, if $M$ is a maximal ideal of $S$, then $a_{\sigma}J(S_M)\subseteq J(S_M)$, hence $a_{\sigma}\in S_M$ for all maximal ideals $M$ of $A$. Therefore $a_{\sigma}\in S$ and $B=A_f$, and so $A_f$ is a maximal $V$-order in $\Sigma_f$.

  (2) If $A_f$ is a semihereditary $V$-order, then it is an extremal $V$-order by \cite[Theorem 1.5]{K1}. Hence $A_f$ is a semihereditary maximal order, and $J(A_f)=J(V)A_f$ by \cite[Corollary 1]{K2}, since $J(V)$ is not a principal ideal of $V$.

  Now suppose $J(A_f)=J(V)A_f$. We will show that $A_f$ is semihereditary. Observe that in this case $J(A_f)$ is the ascending union of projective ideals $\{ vA_f\mid v\in J(V)\}$. Therefore $J(A_f)$ is a flat $A_f$-module, hence $A_f$ is semihereditary by Lemma~\ref{lem:Coherent-lemma}.

  When $A_f$ is semihereditary, then $J(A_f)=J(V)A_f$, hence $J(A_f)=J_G(A_f)$ and $I_{\sigma}=J(V)S$ for every $\sigma\in G$. It follows that $H=G$.

  (3) Let $G_1=\textrm{Gal(}K/L\textrm{)}$. Suppose $A_f$ is semihereditary. Then $G_1\subseteq G=H$, hence the result follows from Lemma~\ref{lem:Subgroup-lemma}.

  (4) If $A_f$ is a valuation ring of $\Sigma_f$, then it is semihereditary and $J(A_f)$ is a maximal ideal of $A_f$. Since $A_f$ is semihereditary, $J(A_f)=J(V)A_f$. Conversely, if $J(V)A_f$ is a maximal ideal of $A_f$, then $J(A_f)=J(V)A_f$, and $A_f$ is a semihereditary primary order in $\Sigma_f$, hence it is a valuation ring of $\Sigma_f$.

  (5) If $A_f$ is semihereditary, then $H=G$, as we have seen above. On the other hand, if $f(\sigma,\tau)\not\in M^2$ for every maximal ideal $M$ of $S$, then $f(\sigma,\tau)\not\in J(S_M)^2$ for every maximal ideal $M$ of $S$.
  Since $J(V)^2=J(V)$, $J(S_M)^2=J(S_M)$, hence $f(\sigma,\tau)\in U(S_M)$ for every maximal ideal $M$ of $S$, and so $f(\sigma,\tau)\in U(S)$. The result now follows from \cite[Theorem 2(a)]{K3}.

  (6) Suppose $\overline{V}$ is a perfect field. If $A_f$ is semihereditary then $H=G$ and, from \cite[Theorem 2(b)]{K3},
  $K/F$ is tamely ramified and defectless. On the other hand, if $H=G$ and $K/F$ is tamely ramified and defectless, then $A_f$ is semihereditary by \cite[Theorem 2(a)]{K3}.

  \end{proof}

Fix a maximal ideal $M$ of $S$. The relation on $G/H$ given by $\sigma H\leq_M\tau H$ if $f(\sigma,\sigma^{-1}\tau)\not\in M$ is well defined and, since $$f(\sigma,\sigma^{-1}\tau)f(\tau,\tau^{-1}\gamma)=f^{\sigma}(\sigma^{-1}\tau,\tau^{-1}\gamma)f(\sigma,\sigma^{-1}\gamma),$$ it is a preorder on $G/H$ (see \cite[Proposition 3.3(a)]{H}). Hence $\leq_M$ induces an equivalence relation on $G/H$ in the usual manner: $\sigma H$ is equivalent to $\tau H$ if $\sigma H\leq_M\tau H$ and $\tau H\leq_M\sigma H$. Further, $\leq_M$ induces a partial ordering on the set of the equivalence classes thus formed. We shall call the partial ordering on the equivalence classes \textit{the graph of} $[f]_M$ and we will denote it by $\textrm{Gr}(f^M)$. We will denote the equivalence class containing $\sigma H$ by $[\sigma]_M$ and, abusing notation, we shall say $[\sigma]_M\leq_M [\tau]_M$ if $\sigma H\leq_M\tau H$. Observe that the graph $\textrm{Gr}(f^M)$ is rooted at $[1]_M$.

The following lemma was originally proved by Wilson when $V$ is a DVR indecomposed in $K$ \cite[Lemma 2.7]{CW2}. The same argument holds in general:

\begin{lemma} \label{lem:Wilson-Lemma}
Let $\sigma,\tau\in G$.
If $v_M$ is a valuation on $K$ corresponding to the maximal ideal $M$ of $S$, then $v_M(f(\sigma,\sigma^{-1}))\leq v_M(f(\tau,\tau^{-1}))$ whenever $\sigma H\leq_M \tau H$.
\end{lemma}

\begin{proof}
This follows from the cocycle identity $$f^{\sigma}(\sigma^{-1}\tau,\tau^{-1})f(\sigma,\sigma^{-1})=f(\sigma,\sigma^{-1}\tau)f(\tau,\tau^{-1}).$$
\end{proof}

If $I$ is an additive subgroup of $\Sigma_f$, we will let $O_l(I)=\{x\in\Sigma_f\mid xI\subseteq I\}$.
We observe that, if $J(V)$ is a principal ideal of $V$ and $I$ is an ideal of $A_f$ containing $J(V)$, then $I$ is finitely generated over $A_f$: $I/J(V)A_f$ is finitely generated over $\overline{V}$, hence there are elements $a_1,a_2,\dots,a_m\in I$ such that $I=a_1V+a_2V+\cdots+a_mV + J(V)A_f$, a finitely generated ideal of $A_f$ since $J(V)$ is a principal ideal of $V$. In the same manner, if $I$ is an ideal of $S$ containing $J(V)$ then $I$ is a finitely generated ideal of $S$.

The following result is quite consequential; in fact, \cite[Theorem]{K5} follows from it.

\begin{theorem} \label{thm:fundamental-theorem}
Suppose $J(V)$ is a principal ideal of $V$. Given the crossed-product order $A_f$,
then $O_l(J_G(A_f))=A_f$ if and only if $f(\sigma,\sigma^{-1})\not\in M^2$ for every $\sigma\in G$ and each maximal ideal $M$ of $S$; if and only if $f(\sigma,\tau)\not\in M^2$ for every $\sigma,\tau\in G$ and each maximal ideal $M$ of $S$.
\end{theorem}

\begin{proof}
Suppose that $O_l(J_G(A_f))=A_f$. Fix a maximal ideal $M$ of $S$.
We will first show that $f(\tau,\tau^{-1})\not\in M^2$ for every $\tau\in G$. By Lemma~\ref{lem:Wilson-Lemma}, it suffices to show that
$f(\sigma,\sigma^{-1})\not\in M^2$ whenever $[\sigma]_M$ is a maximal element with respect to the partial ordering discussed above. So assume $[\sigma]_M$ is a maximal vertex of $\textrm{Gr}(f^M)$.

Observe that, since $M\supseteq J(S)\supseteq J(V)$, both $M$ and $J(S)$ are finitely generated ideals of $S$. Hence they are principal ideals, say $M=rS$ and $J(S)=\pi_SS$, since $S$ is a B\'{e}zout domain. If on the contrary $f(\sigma,\sigma^{-1})\in M^2$, we let $a=\pi_S/r^2$ and $x=ax_{\sigma}$. Then $a\not\in S$, since $a\not\in S_M$.
Hence $x\not\in A_f$.

Fix $\tau\in G$. We want to show that $xI_{\tau}x_{\tau}\subseteq I_{\sigma\tau}x_{\sigma\tau}$.

Suppose $f(\sigma,\tau)\not\in M$ and $f^{\sigma}(\tau,\tau^{-1})\in M$. Since $f(\sigma,\sigma^{-1}(\sigma\tau))=f(\sigma,\tau)\not\in M$, we have $\sigma H\leq_M \sigma\tau H$, and thus
$\sigma\tau H\leq_M\sigma H$ by the choice of $\sigma$. But $f^{\sigma}(\tau,\tau^{-1})\in M$ and
$f^{\sigma}(\tau,\tau^{-1})=f(\sigma,\tau)f(\sigma\tau,\tau^{-1})$. Since $f(\sigma,\tau)\not\in M$, we conclude that $f(\sigma\tau,\tau^{-1})\in M$. But $f(\sigma\tau,\tau^{-1})=f(\sigma\tau,(\sigma\tau)^{-1}\sigma)$ and thus
$\sigma\tau H\not\leq_M\sigma H$, a contradiction.

Therefore $f(\sigma,\tau)\in M$ or $f(\tau,\tau^{-1})\not\in M^{\sigma^{-1}}$, and hence $I_{\tau}^{\sigma}f(\sigma,\tau)\subseteq M$. This implies that $aI_{\tau}^{\sigma}f(\sigma,\tau)\subseteq S$.

We have
\begin{eqnarray*}
(aI_{\tau}^{\sigma}f(\sigma,\tau))f(\sigma\tau,(\sigma\tau)^{-1}) & = &
aI_{\tau}^{\sigma}f^{\sigma}(\tau,\tau^{-1}\sigma^{-1})f(\sigma,\sigma^{-1}) \\ & = &
\pi_S(f(\sigma,\sigma^{-1})/r^2)I_{\tau}^{\sigma}f^{\sigma}(\tau,\tau^{-1}\sigma^{-1})\\ & \subseteq & J(S). \end{eqnarray*} Therefore $aI_{\tau}^{\sigma}f(\sigma,\tau)\subseteq I_{\sigma\tau}$ by Lemma~\ref{lem:the-ideal-lemma}(1), hence
$xI_{\tau}x_{\tau}=aI_{\tau}^{\sigma}f(\sigma,\tau)x_{\sigma\tau}\subseteq I_{\sigma\tau}x_{\sigma\tau}\subseteq J_G(A_f)$. Thus $x\in O_l(J_G(A_f))=A_f$, a contradiction.

We conclude that $f(\tau,\tau^{-1})\not\in M^2$ for every $\tau\in G$ and each maximal ideal $M$ of $S$. From the cocycle identity $f^{\sigma}(\sigma^{-1},\sigma\tau)f(\sigma,\tau)=f(\sigma,\sigma^{-1})$, it follows that $f(\sigma,\tau)\not\in M^2$ for every $\sigma,\tau\in G$, and each maximal ideal $M$ of $S$.

Suppose $f(\sigma,\sigma^{-1})\not\in M^2$ for every $\sigma\in G$ and each maximal ideal $M$ of $S$. Let $x=\sum_{\tau\in G}k_{\tau}x_{\tau}\in O_l(J_G(A_f))$ with $k_{\tau}\in K$. Fix a maximal ideal $M$ of $S$ and let $v_M$ be a valuation on $K$ corresponding to $M$. Observe that if $J(S)=\pi_SS$, then $v_M(\pi_S)$ is the smallest positive element of the value group of $v_M$ hence, if $k\in K$ satisfies $v_M(k)>-v_M(\pi_S)$ then $v_M(k)\geq 0$. Now fix a $\sigma\in G$. The hypothesis implies that $v_M(f(\sigma,\sigma^{-1}))\leq v_M(\pi_S)$. We will show that $v_M(k_{\sigma})\geq 0$.

If $f(\sigma,\sigma^{-1})\in M$ then, since  $k_{\sigma}f(\sigma,\sigma^{-1})I_{\sigma}=
k_{\sigma}x_{\sigma}I_{\sigma^{-1}}x_{\sigma^{-1}}\subseteq I_1=J(S)\subseteq M$ by Lemma~\ref{lem:the-ideal-lemma}(2), we have $v_M(k_{\sigma}f(\sigma,\sigma^{-1}))>0$ as $I_{\sigma}\not\subseteq M$. Hence $v_M(k_{\sigma})>-v_M(f(\sigma,\sigma^{-1}))\geq -v_M(\pi_S)$ and so $v_M(k_{\sigma})\geq 0$.

On the other hand, if $f(\sigma,\sigma^{-1})\not\in M$ then $I_{\sigma}\subseteq M$. Since $J(S)A_f\subseteq J_G(A_f)$, we have $\sigma^{-1}(\pi_S)\in J_G(A_f)$, hence $x\sigma^{-1}(\pi_S)\in J_G(A_f)$ and therefore $k_{\sigma}\pi_S\in I_{\sigma}$. Then $v_M(k_{\sigma}\pi_S)>0$ hence $v_M(k_{\sigma})\geq 0$ and so $k_{\sigma}\in S$ and $x\in A_f$.
\end{proof}

\begin{corollary} \label{cor:First-Maximal-result}
Suppose $J(V)$ is a principal ideal of $V$ and the crossed-product order $A_f$ is a maximal $V$-order in $\Sigma_f$.
Then $f(\sigma,\tau)\not\in M^2$ for every $\sigma,\tau\in G$, and each maximal ideal $M$ of $S$.
\end{corollary}

\begin{proof}
 We note that $O_l(J_G(A_f))$ is a $V$-order in $\Sigma_f$, by \cite[Corollary 1.3]{K1}. Since $O_l(J_G(A_f))\supseteq A_f$ and $A_f$ is a maximal order, we must have equality: $O_l(J_G(A_f))=A_f$.
 \end{proof}

The following proposition has analogies in \cite[Theorem 2.3(3)]{H}, and in a result by Wilson which we shall generalize in Theorem~\ref{thm:Wilson-Theorem} in the next section. In either case, $V$ was a DVR indecomposed in $K$. The approach we have taken below is essentially due to Wilson.

\begin{proposition} \label{prop:chain-graph}
Given the crossed-product order $A_f$, suppose $f(\sigma,\tau)\not\in M^2$ for every $\sigma,\tau\in G$ and each maximal ideal $M$ of $S$.
Then for each maximal ideal $M$ of $S$, the graph $\textrm{Gr}(f^M)$ is a chain.
\end{proposition}

\begin{proof}
Let $M$ be a maximal ideal of $S$, and let $\sigma,\tau\in G$. Write $\sigma=\tau\gamma$. We have $f(\tau,\gamma)f(\tau\gamma,\gamma^{-1})=f^{\tau}(\gamma,\gamma^{-1})$. Since $f(\gamma,\gamma^{-1})\not\in (M^{\tau^{-1}})^2=(M^2)^{\tau^{-1}}$, we have $f(\tau,\gamma)f(\tau\gamma,\gamma^{-1})\not\in M^2$. Therefore $f(\tau,\tau^{-1}(\tau\gamma))=f(\tau,\gamma)\not\in M$ or $f(\tau\gamma,(\tau\gamma)^{-1}\tau)=f(\tau\gamma,\gamma^{-1})\not\in M$, hence
  $\tau H\leq_M \tau\gamma H=\sigma H$, or $\sigma H=\tau\gamma H\leq_M \tau H$, so that the graph of $[f]_M$ is totally ordered.
 \end{proof}

 The converse of the statement above is false (see \cite[Example]{K5}).

 \begin{remark} \label{rem:chain-graph-remark}
 When $J(V)$ is a non-principal ideal of $V$ and $A_f$ is semihereditary, then $H=G$ by Theorem~\ref{thm:The-Main-Nonprincipal-Theorem}, hence for each maximal ideal $M$ of $S$ the graph of $[f]_M$ is trivial. If $J(V)$ is a principal ideal of $V$, then an example of a semihereditary order $A_f$ for which the graph of $[f]_M$ is not a chain for some maximal ideal $M$ of $S$ would be interesting. Obviously $A_f$ cannot be a maximal order and we cannot have $H=G$. As we shall see in the next chapter, the extension $K/F$ cannot be tamely ramified and defectless either. Thus the evidence points to a very pathological scenario indeed, and it may well not exist at all!

 Also, characterizing all crossed-product orders $A_f$ for which the graph of $[f]_M$ is a chain for each maximal ideal $M$ of $S$ would in itself be a worthwhile undertaking.
 \end{remark}

 Unlike in the case of classical crossed-product orders (see for example \cite[Theorem 1]{Hd}), the fact that $A_{f_M}$ or $A_{f_{(M)}}$ is semihereditary for some maximal ideal $M$ of $S$ does not alone guarantee that $A_f$ is semihereditary (see \cite[Example]{K5}). However we have the following affirmative result, but first we need more notation.

Given a maximal ideal $M$ of $S$, let $M=M_1,M_2,\ldots, M_r$ be the
complete list of maximal ideals of $S$, let $U_i=S_{M_i}\cap
K^Z(S_{M_i}\!\!\mid\!\! F)$ with $U=U_1$, and let $(K_i,S_i)$ be a Henselization of
$(K,S_{M_i})$ (see \cite[\S 17]{E} for a definition). Let $(F_h,V_h)$ be the unique Henselization of
$(F,V)$ contained in $(K_1,S_1)$ (see \cite[Theorem 17.11]{E}). We note that $(F_h,V_h)$ is also a
Henselization of $(K^Z,U)$. By \cite[Proposition 11]{HMW}, we have
$S\otimes_VV_h\cong S_1\oplus S_2\oplus\cdots\oplus S_r$.

\begin{proposition} \label{prop:A-global-local-result}
Suppose the crossed-product order $A_f$ is semihereditary (respectively a valuation ring of $\Sigma_f$). Then $A_{f_M}$ is a semihereditary order (respectively a valuation ring of $\Sigma_{f_M}$) for each maximal ideal $M$ of $S$.
\end{proposition}

\begin{proof}
If $A_f$ is semihereditary then its Henselization, $A_f\otimes_VV_h$, is also semihereditary by \cite[Theorem 3.4]{K1}. Since $S\otimes_VV_h\subseteq A_f\otimes_VV_h$, if $e$ is the multiplicative identity of $S_1$ then $e$ is an idempotent of $A_f\otimes_VV_h$, hence $e(A_f\otimes_VV_h)e=\sum_{\sigma\in G^Z}S_1x_{\sigma}\cong A_{f_M}\otimes_UV_h$ is semihereditary. This shows that $A_{f_M}$ is semihereditary by \cite[Theorem 3.4]{K1}.
On the other hand, if $A_f$ is a valuation ring of $\Sigma_f$ then $A_f\otimes_VV_h$ is a valuation ring of
$\Sigma_f\otimes_FF_h$ by \cite[Theorem F]{Wa}, hence $e(A_f\otimes_VV_h)e$ is a valuation ring of $\Sigma_{f_M}\otimes_{K^Z}F_h$ by \cite[\S 1, Theorem 7]{D1}. This shows that $A_{f_M}$ is a valuation ring of $\Sigma_{f_M}$.\end{proof}

The converse of the proposition above does not always hold (see \cite[Example]{K5}). In Section~\ref{sec:graphs} we shall encounter a sufficient condition for the converse to hold.
The author suspects that, if $M$ is a maximal ideal of $S$ and $A_f$ is semihereditary, then $S(G_1)$ is semihereditary for each subgroup $G_1$ of $G$ contained in $G^Z(S_M\!\!\mid\!\! F)$. If that was indeed the case, then $A_{f(M)}$ too would be semihereditary whenever $A_f$ is semihereditary.

  \begin{proposition} \label{prop:fg-case}
  Suppose $S$ is finitely generated over $V$. Then:
  \begin{enumerate}
    \item The crossed-product order $A_f$ is a maximal $V$-order if and only if it is a valuation ring of $\Sigma_f$.
    \item If $J(V)$ is a non-principal ideal of $V$, then $A_f$ is an extremal $V$-order if and only if it is a valuation ring of $\Sigma_f$.
    \item If $J(V)$ is a non-principal ideal of $V$ and $\overline{V}$ is a perfect field, then
    $A_f$ is an extremal $V$-order if and only if it is an Azumaya algebra over $V$.
    \end{enumerate}
  \end{proposition}

  \begin{proof}
  If $S$ is finitely generated over $V$, then $A_f$ is also finitely generated over $V$. By remarks following \cite[Proposition 1.8]{K1}, a finitely generated $V$-order in $\Sigma_f$ is a valuation ring if and only if it is a maximal $V$-order.

  If $A_f$ is extremal then, since it is finitely generated over $V$, it must be a semihereditary order by \cite[Proposition 1.8]{K1}. If $J(V)$ is not a principal ideal of $V$, then $A_f$ is a maximal order by Theorem~\ref{thm:The-Main-Nonprincipal-Theorem} hence it is a valuation ring. Further, $H=G$. If in addition $\overline{V}$ is a perfect field, then $K/F$ is tamely ramified and defectless by \cite[Theorem 2(b)]{K3}. By \cite[18.3 \& 18.6]{E}, we conclude that $K/F$ is unramified and defectless, since $S$ is finitely generated over $V$ and $J(V)$ is not a principal ideal of $V$. Therefore $A_f$ is Azumaya over $V$ by Theorem~\ref{thm:The-Azumaya-Theorem}.
  \end{proof}

  This chapter has dealt with general results about the crossed-product order $A_f$.
  To obtain sharper results however, we need to impose conditions on the extension $K/F$, which we shall do in the following chapters.

  \section{The case when $K/F$ is tamely ramified and defectless}
\label{sec:Tame-Case}

The case which is most amenable to analysis occurs when $K/F$ is tamely ramified and defectless. This is because the key to studying the properties of interest to us of the crossed-product order $A_f$ lies in having a workable description of its Jacobson radical. This is accomplished in the following theorem.

\begin{theorem} \label{thm:The-Jacobson-Radical}
If $K/F$ is tamely ramified and defectless, then $J(A_f)=J_G(A_f)$.
\end{theorem}

\begin{proof}
First, suppose that $V$ is indecomposed in $K$. Then
$$J_G(A_f)=\sum_{\sigma\in H}J(S)x_{\sigma}+\sum_{\sigma\not\in H}Sx_{\sigma}.$$
Hence $A_f/J_G(A_f)=\sum_{\sigma\in H}\overline{S}\overline{x}_{\sigma}$, a semisimple ring by \cite[Theorem 2(a)]{K3},
since $K/L$, where $L$ is the fixed field of $H$, is also a tamely ramified and defectless extension.
Hence $J_G(A_f)=J(A_f)$ if $V$ is indecomposed in $K$.

In general, let $x=\sum_{\gamma}s_{\gamma}x_{\gamma}\in J(A_f)\setminus J_G(A_f)$. Then there exists a $\sigma\in G$, and a maximal ideal $M$ of $S$, such that $s_{\sigma}, f(\sigma,\sigma^{-1})\not\in M$. We fix this $M$. Observe that $xx_{\sigma^{-1}}=\sum_{\gamma}s_{\gamma}f(\gamma,\sigma^{-1})x_{\gamma\sigma^{-1}}\in J(A_f)$ and has constant term $s_{\sigma}f(\sigma,\sigma^{-1})\not\in M$, and hence $xx_{\sigma^{-1}}\not\in J_G(A_f)$.

Thus the set $\{\sum_{\gamma}s_{\gamma}x_{\gamma}\in J(A_f)\setminus J_G(A_f)\mid s_1\not\in M\}$ is non-empty. Let $y=\sum_{\sigma_i}s_{\sigma_i}x_{\sigma_i}$, with $\sigma_1=1$, be a member of this set with the least length. Let $s\in S$ and $j>1$. Then $y'=\sigma_j(s)y-ys=\sum_{i\not=j}(\sigma_j(s)-\sigma_i(s))s_{\sigma_i}x_{\sigma_i}\in J(A_f)$. If $(\sigma_j(s)-s)s_{\sigma_1}\not\in M$, then $y'\in J(A_f)\setminus J_G(A_f)$ has a shorter length than $y$, which is not possible. Therefore $(\sigma_j(s)-s)s_{\sigma_1}\in M$ for every $s\in S$, and so $\sigma_j(s)-s\in M$ for every $s\in S$. Thus, either $s_{\sigma_j}=0$, or $\sigma_j\in G^Z$, and so $y\in S(G^Z)\cap J(A_f)$.

Since $y\in J(A_f)$ and $S(G^Z)\subseteq A_f$, we have $1-z_1yz_2\in U(A_f)$ for all $z_1,z_2\in S(G^Z)$. Whence
$(1-z_1yz_2)^{-1}\in A_f\cap K(G^Z)=S(G^Z)$ for all $z_1,z_2\in S(G^Z)$. This shows that $y\in J(S(G^Z))$.

Observe that $MS(G^Z)$ is an ideal of $S(G^Z)$. Let $$\psi:S(G^Z)\mapsto S(G^Z)/MS(G^Z)$$ be the natural homomorphism. Then $\psi(J(S(G^Z)))\subseteq J(S(G^Z)/MS(G^Z))$, since $\psi$ is onto. But $S(G^Z)/MS(G^Z)$ is finite dimensional over $\overline{V}$. So there is a positive integer $n$ such that
$J(S(G^Z))^n\subseteq MS(G^Z)$.

Let $I=J(S(G^Z))S_M$, an ideal of $S(G^Z)S_M=A_{f_M}$. Then $I^n\subseteq J(S_M)A_{f_M}\subseteq J(A_{f_M})$, hence $I\subseteq J(A_{f_M})$. This shows that
$y\in J(A_{f_M})=\sum_{\gamma\in H_M}J(S_M)x_{\gamma}+\sum_{\gamma\in G^Z\setminus H_M}S_Mx_{\gamma}.$ This contradicts the fact that the constant term of $y$ is a unit in $S_M$. Therefore $J(A_f)=J_G(A_f)$.
\end{proof}

Initially, the author could prove the statement above only in the case when $V$ is indecomposed in $K$. The minimal length argument employed above was first used by Wilson in \cite[Theorem 4.3]{CW2} to show that, when $V$ is a DVR and $K/F$ is tamely ramified, then $J(A_f)=\sum_{\sigma\in G}I_{\sigma}x_{\sigma}$. Instead of ``going down'' to the residue ring $S(G^Z)/MS(G^Z)$, Wilson ``went up'' to the completion of $A_f$ in order to obtain a contradiction.

If $K/F$ is not tamely ramified and defectless, then we may have $J(A_f)\not=J_G(A_f)$, even when $H=G$ (see \cite[Examples 1 and 2]{K3}).

The graph $\textrm{Gr}(f)$ is \textit{lower subtractive}: given $\sigma H\leq \tau H$, we have
$\sigma H\leq \gamma H\leq \tau H$ if and only if $\sigma^{-1}\gamma H\leq \sigma^{-1}\tau H$ (see \cite{H}). This follows from the cocycle identity $f^{\sigma}(\sigma^{-1}\gamma,\gamma^{-1}\tau)f(\sigma,\sigma^{-1}\tau)=
f(\sigma,\sigma^{-1}\gamma)f(\gamma,\gamma^{-1}\tau)$.

The following theorem generalizes \cite[Theorem 2.15]{CW2}, which in turn generalizes \cite[Theorem 2.3]{H}.
In either case, $V$ was a DVR of $F$.

\begin{theorem} \label{thm:Wilson-Theorem}
Suppose $K/F$ is tamely ramified and defectless. Assume $J(V)$ is a principal ideal of $V$, and $V$ is indecomposed in $K$. Then the following are equivalent:
    \begin{enumerate}
    \item The crossed-product order $A_f$ is semihereditary.
    \item $f(\sigma,\tau)\not\in J(S)^2$ for every $\sigma,\tau\in G$.
    \item $H$ is a normal subgroup of $G$ and $G/H$ is cyclic. Further, either $H=G$, or there exists a $\sigma\in G$ such that $G/H=\langle\sigma H\rangle$, the graph $\textrm{Gr}(f)$ is $H\leq\sigma H\leq\sigma^2 H\leq\cdots\leq\sigma^{\mid G/H\mid-1}H$, and $J(A_f)=x_{\sigma}A_f$.
        \end{enumerate}
\end{theorem}

\begin{proof}
Let $v$ be a valuation on $K$ corresponding to $S$ and let $J(S)=\pi_S S$.

$(1) \Rightarrow (2)$.
Suppose $A_f$ is semihereditary. We will first show that $f(\sigma,\sigma^{-1})\not\in J(S)^2$ for all $\sigma\in G$. Fix $\sigma\in G$. By Lemma~\ref{lem:Wilson-Lemma}, it does no harm to assume that $\sigma H$ is a maximal vertex of $\textrm{Gr}(f)$. Suppose $f(\sigma,\sigma^{-1})\in J(S)^2$. Let $x=(1/\pi_S)x_{\sigma}$. We will now show that $x\in O_l(J(A_f))$. Recall that $J(A_f)=\sum_{\gamma\in H}J(S)x_{\gamma}+\sum_{\gamma\not\in H}Sx_{\gamma}$. If $\tau\in H$, then $x\pi_S x_{\tau}=(\sigma(\pi_S)/\pi_S)f(\sigma,\tau)x_{\sigma\tau}\in J(A_f)$, since $\sigma(\pi_S)/\pi_S\in S$ and $\sigma\tau\not\in H$. If $\tau\not\in H$, then we distinguish two cases.

Case 1: $\sigma\tau\in H$. Then $f(\sigma^{-1},\sigma\tau)\in U(S)$ hence, from the identity $$f^{\sigma^{-1}}(\sigma,\tau)f(\sigma^{-1},\sigma\tau)=f(\sigma^{-1},\sigma),$$ we see that $f(\sigma,\tau)\in J(S)^2$, hence $xx_{\tau}=(f(\sigma,\tau)/\pi_S)x_{\sigma\tau}\in J(A_f)$, since $f(\sigma,\tau)/\pi_S\in J(S)$.

Case 2: $\sigma\tau\not\in H$. Then $f(\sigma,\tau)=f(\sigma,\sigma^{-1}(\sigma\tau))\not\in U(S)$, otherwise we would have $\sigma H\leq \sigma\tau H$ hence, by the maximality of $\sigma H$, we would end up with $\sigma H=\sigma\tau H$. This would contradict the fact that $\tau\not\in H$. Thus $xx_{\tau}\in J(A_f)$.

This shows that $x\in O_l(J(A_f))$, contradicting the fact that, since $A_f$ is semihereditary, $O_l(J(A_f))=A_f$ \cite[Theorem 1.5]{K1}. So $f(\sigma,\sigma^{-1})\not\in J(S)^2$ for all $\sigma\in G$. Since $f^{\sigma}(\sigma^{-1},\sigma\tau)f(\sigma,\tau)=f(\sigma,\sigma^{-1})$, $f(\sigma,\tau)\not\in J(S)^2$ for every $\sigma,\tau\in G$.

$(2) \Rightarrow (3)$. Now suppose $f(\sigma,\tau)\not\in J(S)^2$ for every $\sigma,\tau\in G$. First, we will show that $G/H$ is totally ordered. Let $\sigma,\tau\in G$ and write $\sigma=\tau\gamma$. We may assume $\sigma H\not=\tau H$, so that $\gamma\not\in H$. Then $f(\gamma,\gamma^{-1})\not\in U(S)$ hence, since $f(\gamma,\gamma^{-1})\not\in J(S)^2$, we must have $v(f(\gamma,\gamma^{-1}))=v(\pi_S)$. Since $f(\tau,\gamma)f(\tau\gamma,\gamma^{-1})=f^{\tau}(\gamma,\gamma^{-1})$ and $v(\pi_S)$ is the smallest positive element of $\Gamma_S$, one of $f(\tau,\gamma)=f(\tau,\tau^{-1}(\tau\gamma))$, $f(\tau\gamma,\gamma^{-1})=f(\tau\gamma,(\tau\gamma)^{-1}\tau)$ must be in $U(S)$. Therefore either $\tau H\leq \tau\gamma H=\sigma H$, or $\sigma H=\tau\gamma H\leq \tau H$, so that $G/H$ is totally ordered.

We will now show that $H$ is a normal subgroup of $G$. If $\sigma H\leq \tau H$ and $h\in H$, then $h^{-1}H=H\leq\sigma H\leq\tau H$. By lower subtractivity of the graph of $f$, we construe that $h\sigma H\leq h\tau H$. Thus left multiplication by elements of $H$ preserves the ordering on $G/H$. Since $G/H$ is a finite totally ordered set, we must have $h\sigma H=\sigma H$ for every $\sigma\in G$, $h\in H$. This shows that $H$ is a normal subgroup of $G$.

Suppose $H\not=G$. Since $G/H$ is totally ordered, the set $\{\gamma H\mid\gamma\not\in H\}$ has a unique minimal element, say $\sigma H$. We will show that $H\leq \sigma H\leq\sigma^2H\leq\cdots\leq\sigma^iH\leq\cdots\leq\sigma^{r-1}H$, where $r$ is the smallest positive integer for which $\sigma^r\in H$. The proof is by induction on $i$. We know that $H\leq\sigma H$. Suppose $H\leq\sigma H\leq\sigma^2H\leq\cdots\leq\sigma^{i-1}H$ for some $i<r$. Since  $\sigma^{-1}\sigma^{i-1}H=\sigma^{i-2}H\leq\sigma^{i-1}H=\sigma^{-1}\sigma^iH$ by assumption and $\sigma H\leq\sigma^iH$, we conclude by lower subtractivity that $\sigma H\leq \sigma^{i-1}H\leq\sigma^iH$, hence the claim. We fix $\sigma$.

Let $\gamma\in G$. We will now show that $\gamma H\in\{\sigma^iH\mid 0\leq i<r\}$. Since $G/H$ is totally ordered, we can choose $i\in\{1,2,\ldots,r-1\}$ to be the greatest such that $\sigma^iH\leq\gamma H$. If $\sigma^iH\not=\gamma H$, then $\sigma^{-i}\gamma H\not= H$. Hence $\sigma H\leq \sigma^{-i}\gamma H$. Since $\sigma^iH\leq\gamma H$ and $\sigma^{-i}\sigma^{i+1}H=\sigma H\leq\sigma^{-i}\gamma H$, we have $\sigma^iH\leq\sigma^{i+1}H\leq\gamma H$ by the lower subtractivity of the graph of $f$, contradicting the maximality of $i$. So $\gamma H=\sigma^iH$, $G/H=<\sigma H>$, and the graph of $f$ is $$H\leq\sigma H\leq\sigma^2 H\leq\cdots\leq\sigma^{\mid G/H\mid-1}H.$$

Since $\sigma\not\in H$, we have $x_{\sigma}\in J(A_f)$, as $J(A_f)=\sum_{\gamma\in H}J(S)x_{\gamma}+\sum_{\gamma\not\in H}Sx_{\gamma}$. Therefore $x_{\sigma}A_f\subseteq J(A_f)$. We will now show equality. Let $a\in S$. If $\gamma\in H$, then $v(f(\sigma,\sigma^{-1}\gamma))=v(\pi_S)$, hence $b=(a\pi_S/f(\sigma,\sigma^{-1}\gamma))^{\sigma^{-1}}\in S$ and $a\pi_S x_{\gamma}=x_{\sigma}\cdot bx_{\sigma^{-1}\gamma}\in x_{\sigma}A_f$. If $\gamma\not\in H$, then $\sigma H\leq\gamma H$, hence $f(\sigma,\sigma^{-1}\gamma)\in U(S)$ and
$c=(a/f(\sigma,\sigma^{-1}\gamma))^{\sigma^{-1}}\in S$. We have $ax_{\gamma}=x_{\sigma}\cdot cx_{\sigma^{-1}\gamma}\in x_{\sigma}A_f$. Therefore $J(A_f)=x_{\sigma}A_f$.

$(3) \Rightarrow (1)$. If $H=G$, then the result follows from \cite[Theorem 2(a)]{K3}, otherwise $J(A_f)=x_{\sigma}A_f$ is a projective ideal of $A_f$ and the result follows from Lemma~\ref{lem:Coherent-lemma}.
\end{proof}

The gist of the proof of the theorem above is based on the work of Wilson on crossed-product orders over DVRs. His arguments in \cite[Theorem 2.15]{CW2} hold here almost verbatim and we have deliberately preserved them to illustrate the fact that, when $J(V)$ is a principal ideal of $V$ then the theory of these crossed-products resembles the classical theory when $V$ is a DVR. Naturally, there had to be some departures form his approach to cater to this general situation.

An \textit{invariant valuation ring} is a Dubrovin valuation ring of a division algebra $D$ which
is stable under all inner automorphisms of $U(D)$.

\begin{proposition} \label{prop:Cyclic-Division-Algebra}
Suppose the extension $K/F$ is not necessarily tamely ramified and defectless.
Assume $\mid\!\! H\!\!\mid = 1$ and $V$ is indecomposed in $K$. Then the crossed-product order $A_f$ is semihereditary
if and only if the following two conditions hold:
    \begin{enumerate}
    \item $\Sigma_f$ is a cyclic division algebra.
    \item $A_f$ is an invariant valuation ring of $\Sigma_f$.
    \end{enumerate}
If in addition $J(V)$ is a principal ideal of $V$ and $K\not=F$, then the following statements are equivalent:
    \begin{enumerate}
    \item The crossed-product order $A_f$ is semihereditary.
    \item  $f(\gamma,\tau)\not\in J(S)^2$ for every $\gamma,\tau\in G$.
    \item There is a generator $\sigma$ of the group $G$ such that
   $J(A_f)=x_{\sigma}A_f$, and the graph $\textrm{Gr}(f)$ is the chain
    $$1=H\leq\sigma H\leq\sigma^2 H\leq\cdots\leq\sigma^{\mid G\mid-1}H.$$
    \end{enumerate}
\end{proposition}

\begin{proof}
Suppose $J(V)$ is non-principal ideal of $V$. If $A_f$ is semihereditary, then $G=H=1$ by Theorem~\ref{thm:The-Main-Nonprincipal-Theorem} and we are done. From now on, suppose $J(V)$ is a principal ideal of $V$.
The given hypothesis implies that $A_f/J_G(A_f)=\overline{S}$, a field, hence $J(A_f)=J_G(A_f)$ and $A_f$ is primary. If $A_f$ is semihereditary, then it would be a valuation ring of $\Sigma_f$, hence $f(\gamma,\tau)\not\in J(S)^2$ for every $\gamma,\tau\in G$ by Corollary~\ref{cor:First-Maximal-result}. Further, since $A_f$ is an integral Dubrovin valuation ring in this case and the matrix size of $\overline{A}_f$ is one, from \cite[Theorem F, Theorem G, \& Corollary G]{Wa}  we conclude that $\Sigma_f$ must be a division algebra, and $A_f$ is an invariant valuation ring of $\Sigma_f$. The rest of the arguments are as in Theorem~\ref{thm:Wilson-Theorem} above.
\end{proof}

If $W$ is a valuation ring of $F$ such that $V\subsetneqq W$, let $R=WS$. A routine argument shows that $R$ is the integral closure of $W$ in $K$ (see, for example, \cite[Lemma 2.3]{K6}). Let $W_1,W_2,\ldots,W_m$ be all the extensions of $W$ to $K$. Let $V_{ij}$ be all the extensions of $V$ to $K$ such that $WV_{ij}=W_i$, $i=1,2,\ldots, m$. Such $V_{ij}$ exist for each $i$: let $\widetilde{V}'$ be a valuation ring of $\overline{W}_i$ lying over $\widetilde{V}=V/J(W)$. Then $V'=\{t\in W_i\mid t+J(W_i)\in\widetilde{V}'\}$ is a valuation ring of $K$ lying over $V$ such that $WV'=W_i$. Clearly, the $V_{ij}$ are all the extensions of $V$ to $K$, and we have
  $$J(S)=\bigcap_i\bigcap_jJ(V_{ij})\supseteq \bigcap_iJ(W_i)=J(R).$$

We record for later use the fact that $J(R)\subseteq J(S)\subseteq I_{\sigma}$ for each $\sigma\in G$.

Now let $M$ be a maximal ideal of $S$, and let $N=J(WS_M)\cap R$, a maximal ideal of $R$.
Since $J(R_N)\subseteq J(S_M)$,  $G^T(R_N\!\!\mid\!\! F)=\{\sigma\in G\mid \sigma(x)-x\in J(R_N), \forall x\in R_N\}$ is a subgroup of $\{\sigma\in G\mid \sigma(x)-x\in J(S_M), \forall x\in S_M\}=G^T(S_M\!\!\mid\!\! F)$. Therefore, if $V$ is tamely ramified and defectless in $K$ then so is $W$ by \cite[Lemma 1]{K3}: if $0<q:=\textrm{Char(}\overline{W}\textrm{)}$, then $q=\textrm{Char(}\overline{V}\textrm{)}$, hence $\gcd(q,\mid\!\! G^T(R_N\!\!\mid\!\! F)\!\!\mid)=1$ since $G^T(R_N\!\!\mid\!\! F)$ is a subgroup of $G^T(S_M\!\!\mid\!\! F)$.

Denote by $B_f$ the $W$-order $WA_f=\sum_{\sigma\in G}Rx_{\sigma}$ and, for each $\sigma\in G$, let $J_{\sigma}=\bigcap N$, where the intersection is over all the maximal ideals $N$ of $R$ for which $f(\sigma,\sigma^{-1})\not\in N$. Then $J(B_f)=\sum_{\sigma\in G}J_{\sigma}x_{\sigma}$ by Theorem~\ref{thm:The-Jacobson-Radical}.

Of particular interest is the following situation: suppose $J(V)$ is a principal ideal of $V$ and $\textrm{rank(}V\textrm{)}>1$.  Let $P=\bigcap_{n\geq 1}J(V)^n$. Then $P$ is a prime ideal of $V$,
$W=V_P$ is a minimal overring of $V$ in $F$, and $\widetilde{V}=V/J(W)$
is a DVR of $\overline{W}$. We shall henceforth refer to $W$ in this case as the minimal overring of $V$ in $F$.

\begin{theorem} \label{thm:The-Main-Principal-Theorem}
Suppose $K/F$ is tamely ramified and defectless and $J(V)$ is a principal ideal of $V$. Then the following are equivalent:
\begin{enumerate}
    \item $A_f$ is semihereditary.
    \item $A_f$ is an extremal $V$-order.
    \item $O_l(J(A_f))=A_f$.
    \item If $M$ is a maximal ideal of $S$, then $f(\sigma,\sigma^{-1})\not\in M^2$ for all $\sigma\in G$.
    \item If $M$ is a maximal ideal of $S$, then $f(\sigma,\tau)\not\in M^2$ for all $\sigma,\tau\in G$.
    \item $J(A_f)$ is a projective ideal of $A_f$.
     \end{enumerate}
When this happens then, for each maximal ideal $M$ of $S$, the graph $\textrm{Gr}(f^M)$ is a chain.
\end{theorem}

\begin{proof} If $A_f$ is semihereditary $V$-order then it is an extremal $V$-order \cite[Theorem 1.5]{K1}, and if $A_f$ is an extremal $V$-order then $O_l(J(A_f))=A_f$ \cite[Proposition 1.4]{K1}. But $J(A_f)=J_G(A_f)$, by Theorem~\ref{thm:The-Jacobson-Radical}. Therefore conditions (3), (4), and (5) are all equivalent, by Theorem~\ref{thm:fundamental-theorem}. Now we will show that if (3) and (5) are true, then (6) is also true.
The conclusion holds when $V$ is indecomposed in $K$, by  Theorem~\ref{thm:Wilson-Theorem}, so let us assume that
$V$ decomposes in $K$.

Given that $O_l(J(A_f))=A_f$, if the Krull dimension of $V$ is one, then $V$ is a DVR and $A_f$ is hereditary and so $J(A_f)$ is a projective ideal of $A_f$ and we are done. So assume the dimension is greater than one. Let $W$ be the minimal overring of $V$ in $F$. Set $B_f=WA_f$ and $R=WS$, as above. The condition that $f(\sigma,\tau)\not\in M^2$ for all $\sigma,\tau\in G$ and every maximal ideal $M$ of $S$ implies that $f(\sigma,\tau)\in U(WS_M)$ for all $\sigma,\tau\in G$ and every maximal ideal $M$, which in turn implies that $f(\sigma,\tau)\in U(R)$ for all $\sigma,\tau\in G$. Thus $J_{\sigma}=J(R)$ for every $\sigma\in G$, hence $J(B_f)=\sum_{\sigma\in G}J_{\sigma}x_{\sigma}=\sum_{\sigma\in G}J(R)x_{\sigma}\subseteq\sum_{\sigma\in G}I_{\sigma}x_{\sigma}=J(A_f)$.

We first show that the integral closure of $\widetilde{V}=V/J(W)$ in $Z(\overline{B}_f)$ is contained in $A_f/J(B_f)$. So let $y\in B_f$ be an element satisfying the condition that $\overline{y}=y+J(B_f)\in Z(\overline{B}_f)$ and $\overline{y}$ is integral over $\widetilde{V}$. Write $y=\sum_{\sigma\in G}b_{\sigma}x_{\sigma}$ with $b_{\sigma}\in R$. We will show that $y\in A_f$ by showing that $b_{\sigma}\in S_M$ for each maximal ideal $M$ of $S$.

Fix a maximal ideal $M$ of $S$. Let $N=J(WS_M)\cap R$, a maximal ideal of $R$.
Since $\overline{B}_f=\sum_{\sigma\in G}\overline{R}\tilde{x}_{\sigma}$, we may write $\overline{y}=
\sum_{\sigma\in G}\overline{b}_{\sigma}\tilde{x}_{\sigma}$ with $\overline{b}_{\sigma}=b_{\sigma}+J(R)\in\overline{R}$. We have
$\overline{R}=R_1\oplus R_2\oplus\cdots\oplus R_n$, where $R_i=R/N_i$, and $N=N_1,N_2,\ldots,N_n$ are the maximal ideals of $R$. Let $e_i$ be the multiplicative identity of $R_i$. Let $\tau\in G\setminus G^Z(R_N\!\!\mid\!\! F)$.
We want to show that $b_{\tau}\in J(R)$.

Suppose there exists an $i_0$ such that $b_{\tau}\not\in N_{i_0}$. Choose a $\gamma\in G$ satisfying $N^{\gamma}=N_{i_0}$, so that we have $\gamma^{-1}(b_{\tau})\not\in N$. Then, since
$\overline{b}_{\tau}\tilde{x}_{\tau}\in Z(\overline{B}_f)$, we have
$(e_1\tilde{x}_{\gamma^{-1}})(\overline{b}_{\tau}\tilde{x}_{\tau})=
(\overline{b}_{\tau}\tilde{x}_{\tau})(e_1\tilde{x}_{\gamma^{-1}})$, and therefore
$$e_1\gamma^{-1}(\overline{b}_{\tau})\overline{f(\gamma^{-1},\tau)}=
\tau(e_1)\overline{b}_{\tau}\overline{f(\tau,\gamma^{-1})}.$$
This is a contradiction, since $e_1\gamma^{-1}(\overline{b}_{\tau})\overline{f(\gamma^{-1},\tau)}\not= \overline{0}$ but $\tau(e_1)=e_j$ for some $j\not=1$. Therefore $y=\sum_{\sigma\in G^Z(R_N\mid F)}b_{\sigma}x_{\sigma}+ t$, where $t\in J(B_f)$.
Let $y'=\sum_{\sigma\in G^Z(R_N\mid F)}b_{\sigma}x_{\sigma}$ and let $B_{f_M}=WA_{f_M}$. Since $G^Z(R_N\!\!\mid\!\! F)\subseteq G^Z(S_M\!\!\mid\!\! F)$, we see that $y'\in B_{f_M}$; clearly, $y'+J(B_{f_M})$ is an element of $Z(\overline{B}_{f_M})$ integral over $\widetilde{V}$.

By Theorem~\ref{thm:Wilson-Theorem}, $A_{f_M}$ is semihereditary,
since $f_M(\tau,\gamma)\not\in M^2$ for every $\tau,\gamma\in G^Z(S_M\!\!\mid\!\! F)$, hence $A_{f_M}$ is extremal in $\Sigma_{f_M}$. The argument in \cite[Theorem 2]{K2} shows that $y'\in A_{f_M}$. Since $J(R)\subseteq S$, the preceding discussion shows that $b_{\sigma}\in S_M$ for each $\sigma\in G$. Because the choice of the maximal ideal $M$ of $S$ was arbitrary, we conclude that $y\in A_f$ as claimed.

Therefore we have the following: $WA_f=B_f$, $J(B_f)\subseteq J(A_f)$, $O_l(J(A_f))=A_f$, and the integral closure of $\widetilde{V}$ in $Z(\overline{B}_f)$ is contained in $A_f/J(B_f)$. The argument in \cite[Theorem 2]{K2} now shows that $A_f/J(B_f)$ is a hereditary ring, hence $J(A_f)/J(B_f)$ is a projective ideal of $A_f/J(B_f)$. But $B_f=WA_f=(WJ(V))A_f=(J(V)A_f)W\subseteq J(A_f)W\subseteq J(A_f)B_f\subseteq B_f$, hence $J(A_f)B_f=B_f$. The argument in \cite[Lemma 4.11]{M} shows that $J(A_f)$ is a projective ideal of $A_f$.

$(6)\Rightarrow (1)$: This follows from Lemma~\ref{lem:Coherent-lemma}.

The last assertion follows from Proposition~\ref{prop:chain-graph}.
\end{proof}

When $V$ is a DVR, Theorem~\ref{thm:The-Main-Principal-Theorem} takes on a particularly elegant form:

\begin{corollary} \label{cor:The-DVR-Case}
 If $V$ is a DVR and $K/F$ is tamely ramified, then the crossed-product order $A_f$ is hereditary
 if and only if all the values of the two-cocycle $f$ are square-free in $S$.
\end{corollary}

 \begin{remark} \label{rem:Kessler-Wilson}
 Suppose $V$ is a DVR tamely ramified in $K$. Then \cite[Theorem 4.14]{CW2} says that, if for every maximal ideal $M$ of $S$ there exists a ``nice'' set of right coset representatives of $G^Z(S_M\!\!\mid\!\! F)$ in $G$ as described in Section~\ref{sec:intro}, then the crossed-product order $A_f$ is hereditary if and only if $f(\sigma,\tau)\not\in M^2$ for every
 maximal ideal $M$ of $S$, and for every $\sigma,\tau\in G^Z(S_M\!\!\mid\!\! F)$. The presence of a ``nice'' set of right coset representatives is not necessary for $A_f$ to be hereditary (see \cite[Example]{K5}); indeed, Theorem~\ref{thm:The-Main-Principal-Theorem} and Corollary~\ref{cor:The-DVR-Case} say that the precise condition for $A_f$ to be hereditary is that, as $M$ runs through the maximal ideals of $S$, $f(\sigma,\tau)\not\in M^2$ not just for $\sigma,\tau\in G^Z(S_M\!\!\mid\!\! F)$, but rather \textit{for all} $\sigma,\tau\in G$. This small apparent oversight seems to have occurred in \cite[\S 4.3]{Ke} as well. We will revisit this issue later in Corollary~\ref{cor:Kessler-characterization-completed} below.
 \end{remark}

Conditions (1),(2),(3) and (6) of the theorem above again show that the behavior of the crossed-product order $A_f$ when $J(V)$ is a principal ideal of $V$ mirrors that of classical orders over DVRs. We shall pursue this theme throughout this paper. Here is one more piece of evidence for such an assertion, which follows from condition (2) above:

\begin{corollary} \label{cor:Second-Maximal-result}
Suppose $J(V)$ is a principal ideal of $V$ and $K/F$ is tamely ramified and defectless. If the crossed-product order $A_f$ is a maximal $V$-order, then it is semihereditary.
\end{corollary}

The author suspects that the conclusion of the corollary above remains valid when $J(V)$ is a non-principal
ideal of $V$.
In general, even when $J(V)$ is a principal ideal of $V$, a maximal $V$-order in a central simple $F$-algebra need not be semihereditary if it is not a crossed-product order (see \cite[Theorem 5.7]{M}).

\begin{corollary} \label{cor:Subgroup-corollary}
Suppose $K/F$ is tamely ramified and defectless, $A_f$ is semihereditary, and $G_1$ is a subgroup of $G$.
 \begin{enumerate}
      \item If $L$ is the fixed field of $G_1$, and $U$ is a valuation ring of $L$ lying over $V$, then $A_{f_{L,U}}$ is semihereditary.
    \item The ring $S(G_1)$ is semihereditary.
    \end{enumerate}
 \end{corollary}

 \begin{proof} If $J(V)$ is not a principal ideal of $V$, then the statements follow from Theorem~\ref{thm:The-Main-Nonprincipal-Theorem} and Lemma~\ref{lem:Subgroup-lemma};
 otherwise Part (1) follows from Theorem~\ref{thm:The-Main-Principal-Theorem}, and Part (2) follows from Part (1) and \cite[Lemma 2.2]{M2}.
 \end{proof}

Although the converse of Proposition~\ref{prop:A-global-local-result} does not hold even when the extension $K/F$ is tamely ramified and defectless, there is still a ``local-global'' character to the property of being semihereditary: $A_f$ is semihereditary if and only if, for each $\sigma\in G$, $S(\langle\sigma\rangle)$ is semihereditary. In fact, to determine whether or not $A_f$ is semihereditary, it is more efficient to focus only on the maximal vertices of the graph of $f$:

\begin{corollary} \label{cor:A-local-global-result}
Let $\sigma_1,\sigma_2,\ldots,\sigma_m\in G$ be such that the $\sigma_iH$ are all the maximal vertices of the graph $\textrm{Gr}(f)$. If $K/F$ is tamely ramified and defectless then the following statements are equivalent:
\begin{enumerate}
    \item The crossed-product order $A_f$ is semihereditary.
    \item For each $\sigma\in G$, the $\langle\sigma\rangle$-graded ring $S(\langle\sigma\rangle)$ is semihereditary.
    \item For each $i$, the $\langle\sigma_i\rangle$-graded ring  $S(\langle\sigma_i\rangle)$ is semihereditary.
    \item For each maximal ideal $M$ of $S$, $f(\sigma_i,\sigma_i^{-1})\not\in M^2$ for all $i$.
    \end{enumerate}
\end{corollary}

\begin{proof}
If $A_f$ is semihereditary, then $S(\langle\sigma\rangle)$ is semihereditary for each $\sigma\in G$ by Corollary~\ref{cor:Subgroup-corollary}.

If $S(\langle\sigma_i\rangle)$ is semihereditary and $M$ is a maximal ideal of $S$, let $L_i$ be the fixed field of $\sigma_i$, and let $U_i=S_M\cap L_i$. By \cite[Lemma 4.10]{M}, $A_{f_{L_i,U_i}}$ is semihereditary, hence $f(\sigma_i,\sigma_i^{-1})\not\in M^2$.

Now suppose that $\sigma_iH$ are all the maximal vertices of the graph of $f$ and $f(\sigma_i,\sigma_i^{-1})\not\in M^2$ for every maximal ideal $M$ of $S$ and for each $i$. Fix a $\sigma\in G$, and a maximal ideal $M$ of $S$. Choose $\sigma_j$, $1\leq j\leq m$, such that $\sigma H\leq\sigma_jH$. From the cocycle identity
$f^{\sigma}(\sigma^{-1}\sigma_j,\sigma_j^{-1})f(\sigma,\sigma^{-1})=f(\sigma,\sigma^{-1}\sigma_j)f(\sigma_j,\sigma_j^{-1})$, it follows that $f(\sigma,\sigma^{-1})\not\in M^2$, hence $A_f$ is semihereditary.
\end{proof}

The corollary above holds even when $J(V)$ is not a principal ideal of $V$, by Lemma~\ref{lem:Subgroup-lemma} and Theorem~\ref{thm:The-Main-Nonprincipal-Theorem}.
We are now going to give yet another characterization of semihereditary crossed-product orders in Theorem~\ref{thm:A-Secondary-Characterization-Theorem} below, but first we need the following results.

Wilson showed in \cite[Theorem 3.4]{CW2} that, if $V$ is a DVR indecomposed in $K$, then $G^T\subseteq H$ whenever $A_f$ is hereditary. In general, if $A_f$ is semihereditary, then so is $A_{f_M}$ for each maximal ideal $M$ of $S$. Hence, if $V$ is a DVR or $J(V)$ is not a principal ideal of $V$, then $G^T(S_M\!\!\mid\!\! F)\subseteq H_M$ whenever $A_f$ is semihereditary. But the following simple example shows that this is not true in general.

\begin{example} \label{exp:A-Counterexample-to-Wilson}
We give an example of a semihereditary crossed-product order $A_f$ with $V$ indecomposed in $K$ but $G^T\not\subseteq H$.
\end{example}

Let $F=\mathbb{Q}(x,y)$, $K=F(\sqrt{y})$, and $V$ the standard rank-2 valuation ring of $F$ with $J(V)=xV$. Then $G=\{1,\sigma\}=G^T$, where $\sigma$ is defined by $\sigma(\sqrt{y})=-\sqrt{y}$. Define a two-cocycle $f:G\times G\mapsto S^{\#}$ by $f(1,1)=f(1,\sigma)=f(\sigma,1)=1$ and $f(\sigma,\sigma)=x$. Then $A_f$ is semihereditary by Theorem~\ref{thm:The-Main-Principal-Theorem}. By Proposition~\ref{prop:Cyclic-Division-Algebra},
$A_f$ is actually an invariant valuation ring of $\Sigma_f$ since $\mid\!\! H\!\!\mid=1$ and $V$ is indecomposed in $K$.  But $G^T\not\subseteq H$. \qed

If $J(V)=\pi V$ is a principal ideal of $V$, let $V'$ be an extension of $V$ to $K$ with valuation $v$, and let $E=\{\delta\in\Gamma_{V'}\mid 0\leq\delta<v(\pi)\}$. Then $\epsilon=\epsilon(V'\!\!\mid\!\! F)=\mid\!\! E\!\!\mid$ is called the \textit{inertial index} of $V'$ over $F$. If $J(V')=\pi_vV'$, then $v(\pi_v)$ is the smallest positive element of $\Gamma_{V'}$. A routine argument shows that $E=\{0,v(\pi_v),v(\pi_v^2),\ldots,v(\pi_v^{\epsilon-1})\}$ and $v(\pi_v^{\epsilon})=v(\pi)$. Therefore, if $W$ is the minimal overring of $V$ in $F$ and $W'=WV'$,
then $\epsilon(V'\!\!\mid\!\! F)=e(V'/J(W')\!\!\mid\!\!\overline{W})$ $(=[\Gamma_{V'/J(W')}:\Gamma_{V/J(W)}])$.

 Suppose $\epsilon(V'\!\!\mid\!\! F)=e(V'\!\!\mid\!\! F)$. For later use, we record the following two consequences of this assumption. Firstly, $\Gamma_{V'}/\Gamma_V$ must be a cyclic group of order $e(V'\!\!\mid\!\! F)$ with $v(\pi_v)+\Gamma_V$ as a generator. Secondly, if $a\in V$ satisfies $v(a)<e(V'\!\!\mid\!\! F)v(\pi_v)$, then $v(a)<\epsilon(V'\!\!\mid\!\! F)v(\pi_v)=v(\pi^{\epsilon}_v)=v(\pi)$. Hence $v(a)=0$, that is, $a\in U(V)$. (We shall make use of the first observation in Proposition~\ref{prop:A-Generalized-Janusz-Result} of Section~\ref{sec:Dubrovin-Case}, and of the second observation in Remark~\ref{rem:alternate-to-Wilson} below.)

Suppose $V_1$ is a valuation ring of $K$ lying over $V$. Then the extension $K/F$ is called \textit{totally ramified} if $[K:F]=[\Gamma_{V_1}:\Gamma_V]$. The extension $K/F$ is called \textit{tame totally ramified} if it is both totally ramified and tamely ramified. Of course when this happens, $V$ is defectless in $K$.

\begin{lemma} \label{lem:A-generalized-Wilson-Theorem}
Suppose that $K/F$ is tame totally ramified and $S$ is finitely generated over $V$. If $A_f$ is semihereditary, then $H=G$.
\end{lemma}

\begin{proof} The lemma is true when $J(V)$ is not a principal ideal of $V$, by Theorem~\ref{thm:The-Main-Nonprincipal-Theorem}. From now on, assume that
$J(V)$ is a principal ideal of $V$. If the Krull dimension of $V$ is one, then $V$ is a DVR and the lemma is a result of Wilson \cite[Theorem 3.4]{CW2}.
For the sake of completeness, we now present a slight variant of his argument: suppose the conclusion does not hold. By Theorem~\ref{thm:Wilson-Theorem}, there exists a $\sigma\in G$ such that $f(\sigma,\sigma^{-1})\not\in U(S)$, $G/H=\langle\sigma H\rangle$, and the graph of $f$ is $$H\leq\sigma H\leq\sigma^2 H\leq\cdots\leq\sigma^{\mid G/H\mid-1}H.$$

If $F'$ is the fixed field of $\sigma$, $G'=\textrm{Gal(}K/F'\textrm{)}=\langle\sigma\rangle$, $f'=f_{\mid_{G'\times G'}}$, and $H'$ is the subgroup of $G'$ associated to $f'$, then $H'=H\cap\langle\sigma\rangle$. Since $G/H=\langle\sigma H\rangle$, we have $\langle\sigma \rangle/(\langle\sigma\rangle\cap H)\cong G/H$. Therefore the graph of $f'$ is $$H'\leq\sigma H'\leq\sigma^2 H'\leq\cdots\leq\sigma^{\mid G'/H'\mid-1}H'.$$ The extension $K/F'$ is also tame totally ramified. Thus it does no harm to assume that $G=\langle\sigma\rangle$, and then show that we cannot have $f(\sigma,\sigma^{-1})\not\in U(S)$. Let $n=\mid \!\!G\!\!\mid$, and $r=\mid\!\! G/H\!\!\mid$. Since $H\not=G$, we have $r>1$.

First assume that $(F,V)$ is complete. By \cite[Propositions 3.1]{W}, there exists $\pi_S\in S$, and an $n^{\textrm{th}}$-root of unity $\omega\in V$, such that $J(S)=\pi_SS$ and $\sigma(\pi_S)=\omega\pi_S$. The argument in
\cite[Proposition 2.1]{W} shows that $\overline{\omega}$ is a \textit{primitive} $n^{\textrm{th}}$-root of unity in $\overline{V}$.

We have $(x_{\sigma})^r=f(\sigma,\sigma)f(\sigma^2,\sigma)\cdots f(\sigma^{r-1},\sigma)x_{\sigma^r}$.
Since $$H\leq\sigma H\leq\sigma^2H\leq\cdots\leq\sigma^{r-1}H,$$ we have  $$f(\sigma,\sigma)=f(\sigma,\sigma^{-1}\sigma^2),f(\sigma^2,\sigma)=f(\sigma^2,\sigma^{-2}\sigma^3),\ldots,$$ $$
f(\sigma^{r-2},\sigma)=f(\sigma^{r-2},\sigma^{2-r}\sigma^{r-1})\in U(S).$$ On the other hand, if $f(\sigma^{r-1},\sigma)\in U(S)$ then we would have $x_{\sigma}\in U(A_f)$, a contradiction. Since $A_f$ is hereditary, we have $f(\sigma^{r-1},\sigma)\not\in J(S)^2$ and so $f(\sigma^{r-1},\sigma)=z\pi_S$ for some $z\in U(S)$. So $(x_{\sigma})^r=t\pi_Sx_{\sigma^r}$ for some $t\in U(S)$.  Therefore $(x_{\sigma})^n=((x_{\sigma})^r)^{n/r}=(t\pi_Sx_{\sigma^r})^{n/r}=\omega^m\pi_S^ku$ for some $u\in U(S)$, $k=n/r$, and $m$ some integer. Since $r>1$, we have $1\leq k<n$. Then $$(x_{\sigma})^{n+1}=(x_{\sigma})^nx_{\sigma}=\omega^m\pi_S^kux_{\sigma}$$ and  $$(x_{\sigma})^{n+1}=x_{\sigma}(x_{\sigma})^n=x_{\sigma}\omega^m\pi_S^ku=\omega^{m+k}\pi_S^k\sigma(u)x_{\sigma}.$$
Therefore $\omega^m\pi_S^ku=\omega^{m+k}\pi_S^k\sigma(u)$ whence $\omega^k=u/\sigma(u)$ and therefore  $\overline{\omega}^k=\overline{1}$ in $\overline{V}$, contradicting the fact that $\overline{\omega}$ is a primitive $n^{\textrm{th}}$-root of unity in $\overline{V}$ and $1\leq k< n$. We should therefore always have $H=G$ when $(F,V)$ is complete.

If $(F,V)$ is not complete, let $(\widehat{F},\widehat{V})$ be the completion of $(F,V)$. Since $V$ is indecomposed in $K$,
$\widehat{K}=K\widehat{F}$ is a field, and $\widehat{S}=S\otimes_V\widehat{V}$ is a valuation ring of $\widehat{K}$ lying over $\widehat{V}$. Further, the extension $\widehat{K}/\widehat{F}$ has the same group $G$ as $K/F$ and satisfies the hypothesis of this lemma. Let $\widehat{f}:G\times G\mapsto \widehat{S}^{\#}$ be given by $\widehat{f}(\sigma,\tau)=f(\sigma,\tau)$ for all $\sigma,\tau\in G$. Then $A_{\widehat{f}}=\sum_{\sigma\in G}\widehat{S}x_{\sigma} (\cong A_f\otimes_V\widehat{V})$ is hereditary, hence $\widehat{f}(\sigma,\sigma^{-1})\in U(\widehat{S})\cap S=U(S)$ for all $\sigma\in G$. Therefore $H=G$ as claimed.

Now suppose that the Krull dimension of $V$ is greater than one. Let $W$ be the minimal overring of $V$ in $F$, let $R=WS$ and $B_f=WA_f$ as before. Since $A_f$ is semihereditary, $f(\sigma,\tau)\not\in J(S)^2$ for all $\sigma,\tau\in G$, hence $f(\sigma,\tau)\in U(R)$ for all $\sigma,\tau\in G$, since $J(R)\subseteq J(S)^2$. Then
$$\epsilon(S\!\!\mid\!\! F)=e(S/J(R)\!\!\mid\!\!\overline{W})\leq [\overline{R}:\overline{W}]\leq [K:F]=e(S\!\!\mid\!\! F).$$ Because $S$ is finitely generated over $V$, we have $\epsilon(S\!\!\mid\!\! F)=e(S\!\!\mid\!\! F)$ by \cite[Theorem 18.6]{E}, hence
$[\overline{R}:\overline{W}]=[K:F]$ and so $e(R\!\!\mid\!\! F)=1$. But $W$ is tamely ramified and defectless in $K$.
Therefore $W$ is unramified and defectless in $K$, $\textrm{Gal(}\overline{R}/\overline{W}\textrm{)}=G$, and $B_f$ is Azumaya over $W$ by Theorem~\ref{thm:The-Azumaya-Theorem}. Further, if $\overline{f}(\sigma,\tau) := f(\sigma,\tau)\, \textrm{mod}\, J(R)$ for all $\sigma,\tau\in G$, then $\overline{f}:G\times G\mapsto (S/J(R))^{\#}$ is a normalized two-cocycle, and $A_f/J(B_f)=A_{\overline{f}}$ is a crossed-product $V/J(W)$-algebra in $\Sigma_{\overline{f}}=\overline{B}_f=(\overline{R}/\overline{W},G,\overline{f})$.
 Observe that $Z(A_{\overline{f}})=V/J(W)$ is a DVR tame totally ramified in $\overline{R}$ as $e(S/J(R)\!\!\mid\!\!\overline{W})= [\overline{R}:\overline{W}]$ and $\textrm{Char(}\overline{V}\textrm{)}\nmid\,\mid\!\! G\!\!\mid$. Further, $A_{\overline{f}}$ is hereditary by Theorem~\ref{thm:Wilson-Theorem}. We conclude that $G=H$ as claimed.
 \end{proof}

 \begin{remark} \label{rem:alternate-to-Wilson}
 Here is an alternate and quick way of proving Lemma~\ref{lem:A-generalized-Wilson-Theorem}: given the assumptions of the lemma, let $\sigma\in G$ be of order $n>1$, let $L$ be the fixed field of $\sigma$, and let $U=S\cap L$. The assumptions given imply that $\epsilon(S\!\!\mid\!\! L)=e(S\!\!\mid\!\! L)=n$ and, if $J(S)=\pi_S S$ and $v$ is a valuation on $K$ corresponding to $S$, then the following holds: if $u\in U$ satisfies $v(u)<nv(\pi_S)$ then $v(u)=0$. Now let $a=\prod_{i=1}^{n-1}f(\sigma^i,\sigma)$. Then $a\in U$, by \cite[Theorem 30.3]{R}. Since $A_f$ is semihereditary, $v(f(\sigma^i,\sigma))\leq v(\pi_S)$, hence $v(a)=\sum_{i=1}^{n-1}v(f(\sigma^i,\sigma))\leq (n-1)v(\pi_S)<nv(\pi_S)$. Hence $a\in U(S)$, and thus $f(\sigma^{-1},\sigma)=f(\sigma^{n-1},\sigma)\in U(S)$. This shows that $H=G$.

However, this alternate approach above does not build on the techniques already in place when $V$ is a DVR, and hence does not quite illuminate the close connection we want to maintain between our theory and the classical theory. We shall, nonetheless, use the trick above later in Proposition~\ref{prop:Harada-result}.
\end{remark}

 \begin{proposition} \label{prop:A-generalized-Wilson-Proposition}
 Assume $K/F$ is tamely ramified and defectless, $S$ is finitely generated over $V$, and $A_f$ is a semihereditary order. Let $M$ be a maximal ideal of $S$. Then:
 \begin{enumerate}
    \item $G^T(S_M\!\!\mid\!\! F)\subseteq H_M$. In particular, if $V$ is indecomposed in $K$ then $G^T\subseteq H$.
    \item If $K/F$ is an abelian extension, then $G^T(S_M\!\!\mid\!\! F)\subseteq H$.
    \end{enumerate}
 \end{proposition}

 \begin{proof}
 If $J(V)$ is a not a principal ideal of $V$, then $H=G$ by Theorem~\ref{thm:The-Main-Nonprincipal-Theorem} and we are done. So let us assume that $J(V)$ is a principal ideal of $V$.

 (1) We know that $A_{f(M)}$ is semihereditary by Corollary~\ref{cor:Subgroup-corollary}. By Lemma~\ref{lem:A-generalized-Wilson-Theorem}, $f(\sigma,\sigma^{-1})\not\in M$ for every $\sigma\in G^T$, hence $G^T\subseteq H_M$.

 (2) If $G$ is abelian then, as $M$ varies over the set of maximal ideals of $S$, the corresponding inertial groups coincide. By part (1), if $\sigma\in G^T$ then $f(\sigma,\sigma^{-1})\not\in N$ for each maximal ideal $N$ of $S$, and hence $f(\sigma,\sigma^{-1})\in U(S)$.
  \end{proof}

\begin{theorem} \label{thm:A-Secondary-Characterization-Theorem}
Suppose $K/F$ is tamely ramified and defectless and $S$ is finitely generated over $V$. Then the following statements are equivalent:
\begin{enumerate}
    \item The crossed-product order $A_f$ is semihereditary.
     \item For each maximal ideal $M$ of $S$, we have $G^T(S_M\!\!\mid\!\! F)\subseteq H_M$ and there exists a set of right coset representatives $\sigma_1,\sigma_2,\ldots,\sigma_m$ of $G^T(S_M\!\!\mid\!\! F)$ in $G$ such that for all i, $f(\sigma_i,\sigma_i^{-1})\not\in M^2$.
     \item For each maximal ideal $M$ of $S$, $A_{f_{(M)}}$ is semihereditary and there exists a set of right coset representatives $\sigma_1,\sigma_2,\ldots,\sigma_m$ of $G^T(S_M\!\!\mid\!\! F)$ in $G$ such that for all i, $f(\sigma_i,\sigma_i^{-1})\not\in M^2$.
     \end{enumerate}
\end{theorem}

\begin{proof} $(1) \Leftrightarrow (2)$: Suppose $A_f$ is semihereditary. Let $M$ be a maximal ideal of $S$, and let $\sigma_1,\sigma_2,\ldots,\sigma_m$ be any set of right coset representatives of $G^T$ in $G$.
Then $G^T\subseteq H_M$ by Proposition~\ref{prop:A-generalized-Wilson-Proposition}, and $f(\sigma_i,\sigma_i^{-1})\not\in M^2$ by Theorem~\ref{thm:The-Main-Nonprincipal-Theorem} (respectively Theorem~\ref{thm:The-Main-Principal-Theorem}) if $J(V)^2=J(V)$ (respectively if $J(V)^2\not= J(V)$).

Conversely, let $M$ be a maximal ideal of $S$.
Suppose $G^T\subseteq H_M$ and $\sigma_1,\sigma_2,\ldots,\sigma_m$ is a set of right coset representatives of $G^T$ in $G$ such that $f(\sigma,\sigma^{-1})\not\in M^2$ for each $i$. Let $\sigma\in G$. We will show that $f(\sigma,\sigma^{-1})\not\in M^2$.

There exists an $i$, $0\leq i\leq m$, and an $h\in G^T\subseteq H_M$, such that $\sigma=h\sigma_i$. From the cocycle identity
$$f^{h^{-1}}(h\sigma_i,\sigma_i^{-1}h^{-1})f^{h^{-1}}(h,\sigma_i)f^{\sigma_i}(\sigma_i^{-1},h^{-1})=
f(h^{-1},h)f(\sigma_i,\sigma_i^{-1}),$$

we conclude that $f^{h^{-1}}(\sigma,\sigma^{-1})\not\in M^2$ and, since $M^h=M$, we have $f(\sigma,\sigma^{-1})\not\in M^2$. It follows from Theorem~\ref{thm:The-Main-Nonprincipal-Theorem} and Theorem~\ref{thm:The-Main-Principal-Theorem} that $A_f$ is semihereditary.

 $(2) \Leftrightarrow (3)$: This now follows from Proposition~\ref{prop:A-generalized-Wilson-Proposition} and \cite[Theorem 2]{K3}.
\end{proof}

We end this section by highlighting further differences and similarities between our crossed-product orders, and the ones already in the literature.

The cocycle $f$ is said to be \textit{nullcohomologous over $S$} (respectively \textit{nullcohomologous over K}) if $f\! \sim_S\! 1$ (respectively $f\! \sim_K\! 1$), where $1$ denotes the trivial cocycle that takes on the constant value $1$. We hasten to point out that we can have $f\! \sim_K\! 1$ \textit{without} having $f\! \sim_S\! 1$, which is precisely why the following proposition supplants \cite[Theorem 1]{K3}.

It is well known that, if $V$ is a DVR, then $K/F$ is tamely ramified if and only if there is an $s\in S$ such that $\sum_{\sigma\in G}\sigma(s)=1$ \cite[Theorem 3.2]{AR}. This criterion is fully generalizable by \cite[Corollary 1]{K3}: if $V$ is an arbitrary valuation ring of $F$, then the Galois extension $K/F$ is tamely ramified and defectless if and only if $\sum_{\sigma\in G}\sigma(s)=1$ for some $s\in S$. We now have the following result, which generalizes a classical result by Auslander and Rim \cite[Corollary 3.6]{AR} and is obtained by modifying a little an argument by Rosen (see \cite[Theorem 40.13]{R}).

\begin{proposition} \label{prop:Auslander-Rim-result}
Suppose $f$ is nullcohomologous over $K$, but not necessarily nullcohomologous over $S$. Then $A_f$ is semihereditary if and only if $K/F$ is tamely ramified and defectless, and for all $\sigma,\tau\in G$ and every maximal ideal $M$ of $S$, $f(\sigma,\tau)\not\in M^2$.
\end{proposition}

\begin{proof}
Since $f\! \sim_K\! 1$, there exists $k_{\sigma}\in K^{\#}$ such that, if $y_{\sigma}=k_{\sigma}x_{\sigma}$, then
$y_{\sigma}y_{\tau}=y_{\sigma\tau}$ for all $\sigma,\tau\in G$. Suppose $A_f$ is semihereditary. Let $x=\sum_{\sigma\in G}y_{\sigma}$. Let $\phi:A_f\mapsto A_fx:\phi(t)=tx$. Since $A_f$ is a semihereditary order in $\Sigma_f$, $A_fx$ is a projective left $A_f$-module, hence there exists an $A_f$-module homomorphism $\psi:A_fx\mapsto A_f$ such that $\phi\psi=1$. Let $\psi(x)=\sum_{\sigma\in G}s_{\sigma}x_{\sigma}$ with $s_{\sigma}\in S$. Fix a $\tau\in G$. Since $y_{\tau}x=x$, we have $\psi(y_{\tau}x)=\psi(x)=\sum_{\sigma\in G}s_{\sigma}x_{\sigma}$. On the other hand, $\psi(y_{\tau}x)=\psi(k_{\tau}x_{\tau}x)$. Since $\psi$ is $S$-linear and $K$ is the quotient field of $S$, we have $\psi(k_{\tau}x_{\tau}x)=k_{\tau}\psi(x_{\tau}x)$ in $\Sigma_f$, hence $$\psi(k_{\tau}x_{\tau}x)=k_{\tau}x_{\tau}\psi(x)=k_{\tau}x_{\tau}\sum_{\sigma\in G}s_{\sigma}x_{\sigma}.$$

Therefore $\sum_{\sigma\in G}s_{\sigma}x_{\sigma}=\sum_{\sigma\in G}k_{\tau}\tau(s_{\sigma})f(\tau,\sigma)x_{\tau\sigma}$. Hence  $s_{\tau\sigma}=k_{\tau}\tau(s_{\sigma})f(\sigma,\tau)$ and we get $$s_{\tau}=k_{\tau}\tau(s_1).$$
Therefore $x=\phi\psi(x)=\phi(\sum_{\sigma\in G}s_{\sigma}x_{\sigma})=(\sum_{\sigma\in G}s_{\sigma}x_{\sigma})x=
(\sum_{\sigma\in G}k_{\sigma}\sigma(s_1)x_{\sigma})x=\sum_{\sigma\in G}\sigma(s_1)(k_{\sigma}x_{\sigma}x)=
\sum_{\sigma\in G}\sigma(s_1)(y_{\sigma}x)=(\sum_{\sigma\in G}\sigma(s_1))x$.

 We conclude that $\sum_{\sigma\in G}\sigma(s_1)=1$, and thus $K/F$ is tamely ramified and defectless. The rest follows from Theorem~\ref{thm:The-Main-Nonprincipal-Theorem} and Theorem~\ref{thm:The-Main-Principal-Theorem}.

\end{proof}

Unfortunately, when $K/F$ is wildly ramified (that is, $K/F$ is not tamely ramified) and $f\! \not\sim_K\! 1$, then there seems to be no simple criterion in general for determining when the crossed-product order $A_f$ is semihereditary, other than the one contained in Lemma~\ref{lem:Coherent-lemma}, even when $H=G$. In \cite[Example 2]{K3}), we see a hereditary crossed-product order $A_f$ where $V$ is a DVR, $H=G$, and $K/F$ is wildly ramified. Thus the case when $K/F$ is wildly ramified is well worth exploring.

Now suppose the extension $K/F$ is not necessarily tamely ramified and defectless. Let $M$ be a maximal ideal of $S$. We let $G^V(S_M\!\!\mid\!\! F)=\{\sigma\in G\mid 1-\frac{\sigma(x)}{x}\in J(S_M)\forall x\in K^{\#}\}$, the \textit{ramification group of $S_M$ over F}. We shall denote it simply by $G^V$ if the context is clear. If $\overline{p}$ is the characteristic exponent of $\overline{V}$, then $G^V$ is the only Sylow $\overline{p}$-subgroup of $G^T$ \cite[Theorem 20.18]{E}. Therefore by \cite[Lemma 2]{K3}, the Galois extension $K/F$ is tamely ramified and defectless if and only if $G^V=\{1\}$.

The following result generalizes a result of Harada (see \cite[Theorem 2]{Hd}).

\begin{proposition} \label{prop:Harada-result}
Suppose rank$(V)=1$ or $J(V)$ is non-principal ideal of $V$. Assume $\overline{V}$ is a perfect field. Then $A_f$ is semihereditary if and only if $K/F$ is tamely ramified and defectless, and for all $\sigma,\tau\in G$ and every maximal ideal $M$ of $S$, $f(\sigma,\tau)\not\in M^2$.
\end{proposition}

\begin{proof}
If $J(V)$ is a non-principal ideal of $V$ then the result was already covered in Theorem~\ref{thm:The-Main-Nonprincipal-Theorem}. From now on, let us assume that $V$ is a DVR. Suppose $A_f$ is hereditary. Let $M$ be a maximal ideal of $S$. We will show that $G^V=\{1\}$. Since $G^V\subseteq G^Z$ and $A_{f_M}$ is also hereditary by Proposition~\ref{prop:A-global-local-result}, we can, without loss of generality, assume that $V$ is indecomposed in $K$.

Suppose $H\cap G^V\not=\{1\}$. By Lemma~\ref{lem:Subgroup-lemma}(2), $S(H\cap G^V)$ is hereditary, which contradicts Harada's result \cite[Theorem 2]{Hd}. Thus $H\cap G^V=\{1\}$, and so if $L$ is the fixed field of $H$ then $K/L$ is tamely ramified and defectless, hence $J(S(H))=\sum_{\sigma\in H}J(S)x_{\sigma}$ by Theorem~\ref{thm:The-Jacobson-Radical}. Therefore, since
$J_G(A_f)=\sum_{\sigma\in H}J(S)x_{\sigma}+\sum_{\sigma\in G\setminus H}Sx_{\sigma}$, we conclude that $A_f/J_G(A_f)$ is semisimple, hence $J_G(A_f)=J(A_f)$. But since $A_f$ is hereditary, $O_l(J(A_f))=A_f$ hence, by Theorem~\ref{thm:fundamental-theorem}, $f(\sigma,\tau)\not\in J(S)^2$ for every $\sigma,\tau\in G$.

 If $\mid\!\! G^V\!\!\mid > 1$, let $\sigma\in G^V$ be an element of order $\overline{p}$, with fixed field $E$. Let $U=S\cap E$. Since $\overline{V}$ is a perfect field, we have $\overline{U}=\overline{S}$, by \cite[Theorem 19.11]{E}, hence $e(S\!\!\mid\!\! E)=\overline{p}$.

Let $a=\prod_{i=1}^{\overline{p}-1}f(\sigma^i,\sigma)\in U$. As was the case in Remark~\ref{rem:alternate-to-Wilson}, we see that $a\in U(S)$ hence $\sigma\in H$, a contradiction. This shows that $G^V=\{1\}$ hence $K/F$ is tamely ramified.

\end{proof}

If rank$(V)>1$ and $J(V)$ is a principal ideal of $V$, then the proposition above does not hold, even when $H=G$ (see \cite[Examples 1 \& 2]{K3}).

A characterization of what is referred to as a ``hereditary global crossed product order'' in \cite{Ke} is now realized in the following corollary. For purposes of the corollary, $S$ will have a slightly different meaning from the one in use in the rest of the paper. The corollary gives the converse of
\cite[Theorem 4.9]{Ke}, and is obtained through central localization of the order $\sum_{\sigma\in G}Sx_{\sigma}$ by virtue of \cite[Theorem 3.24]{R}.

\begin{corollary} \label{cor:Kessler-characterization-completed}
Let $F$ be an algebraic number field, $K/F$ a finite Galois extension with group $G$, $S$ the Dedekind ring of $K$,
and $f:G\times G\mapsto S^{\#}$ a normalized two-cocycle. Then $\sum_{\sigma\in G}Sx_{\sigma}$ is a hereditary order in the crossed-product algebra $\sum_{\sigma\in G}Kx_{\sigma}$ if and only if for every maximal ideal $M$ of $S$, $S_M$ is tamely ramified over $F$ and for all $\sigma,\tau\in G$, $f(\sigma,\tau)\not\in M^2$.
\end{corollary}

\section{Dubrovin crossed-product orders}
\label{sec:Dubrovin-Case}

We now turn our attention to determining when the crossed-product order $A_f$ is a valuation ring of $\Sigma_f$.
Recall that the crossed-product order $A_f$ is a valuation ring of $\Sigma_f$ if and only if it is semihereditary and primary. We will therefore begin by characterizing primary crossed-product orders.
We will for the most part assume that $K/F$ is tamely ramified and defectless in this section, as we did in most of the previous section, to take full advantage of the theory developed thus far.

\begin{lemma} \label{lem:Primarity-Lemma}
Suppose $K/F$ is tamely ramified and defectless.
\begin{enumerate}
    \item The crossed-product order $A_f$ is primary if and only if, for some maximal ideal $M$ of $S$, $A_{f_M}$ is primary and there exists a set of right coset representatives $\sigma_1,\sigma_2,\ldots,\sigma_r$ of $G^Z(S_M\!\!\mid\!\! F)$ in $G$ such that for all $i$, $f(\sigma_i,\sigma_i^{-1})\not\in M$.
    \item If $M$ is a maximal ideal of $S$, then the ring $A_{f_M}$ is primary if and only if $A_{f_{(M)}}$ is primary.
    \item If in addition $H=G$, then the following are equivalent:
        \begin{enumerate}
        \item $A_f$ is primary.
        \item $A_{f_M}$ is primary for some maximal ideal $M$ of $S$.
        \item $A_{f_{(M)}}$ is primary for some maximal ideal $M$ of $S$.
        \end{enumerate}
    \end{enumerate}
\end{lemma}

\begin{proof} Suppose $A_f$ is primary and let $M$ be a maximal ideal of $S$. Then $J(A_f)$ is a maximal ideal of $A_f$. Let $\widehat{M}=\bigcap N$, where the intersection ranges over all maximal ideals $N$ of $S$ different from $M$. Then $A_f\widehat{M}A_f=A_f$, since $\widehat{M}\supsetneqq J(S)$ hence $A_f\widehat{M}A_f \nsubseteqq J(A_f)$. The argument in \cite[Lemma 2.7]{K6} (which was an adaptation of an argument in \cite[Theorem 3.2]{H}) shows that there is a set of right coset representatives $\sigma_1,\sigma_2,\ldots,\sigma_r$ of $G^Z$ in $G$ such that $f(\sigma_i,\sigma_i^{-1})\not\in M$ for all $i$. In general, whenever one has such a set of coset representatives of $G^Z$ in $G$, we can use the argument in \cite[Theorem 3.12]{H} to show that $\overline{A}_f\cong M_r(\overline{A}_{f_M})$.

Observe that $\overline{A}_{f_M}=\overline{S}_M(H_M)=\overline{S_M(H_M)}$, hence $A_{f_M}$ is primary if and only if $S_M(H_M)$ is primary. By \cite[Proposition 1]{K4}, $S_M(H_M)$ is primary if and only if $S_M(G^T\cap H_M)$ is primary. The result now follows from the fact that $\overline{A}_{f_{(M)}}=\overline{S}_M(G^T\cap H_M)=\overline{S_M(G^T\cap H_M)}$.
\end{proof}

Part (3) of the lemma above originally appeared in \cite[Lemma 1 \& Proposition 1]{K4}. We have included it here for the sake of completeness.

Recall that, if $R$ is a ring, $G_1$ is a group acting trivially on $R$, and $g:G_1\times G_1\mapsto U(R)$ is a two-cocycle, then the crossed-product of $G_1$ over $R$ is called a \textit{twisted group ring}, denoted by $R^t[G_1]$ (see \cite{P}).

Lemma~\ref{lem:Primarity-Lemma} above and its proof immediately yields the following.

\begin{proposition} \label{prop:Valuation-Ring-Proposition}
Suppose $K/F$ is tamely ramified and defectless.
\begin{enumerate}
\item Let $M$ be a maximal ideal of $S$.
\begin{enumerate}
    \item If $A_f$ is a valuation ring of $\Sigma_f$, then $A_{f_M}$ is a valuation ring of $\Sigma_{f_M}$.
    \item If $A_{f_M}$ is a valuation ring of $\Sigma_{f_M}$, then $A_{f_{(M)}}$ is a valuation ring of $\Sigma_{f_{(M)}}$.
\end{enumerate}
\item If in addition we also assume that $V$ is indecomposed in $K$ and $A_f$ is semihereditary,
then the following are equivalent:
\begin{enumerate}
    \item The crossed-product order $A_f$ is a valuation ring of $\Sigma_f$.
    \item The crossed-product order $S(H)$ is a valuation ring of $K(H)$.
    \item The crossed-product order $S(H\cap G^T)$ is a valuation ring of $K(H\cap G^T)$.
    \item The crossed-product order $A_{f_{(M)}}$ is a valuation ring of $\Sigma_{f_{(M)}}$.
    \item The twisted group ring $\overline{S}^t[H\cap G^T]$ is simple.
    \item The $H$-graded ring $\overline{S}(H)$ is simple.
    \end{enumerate}
\end{enumerate}
\end{proposition}

\begin{corollary} \label{cor:The-Schur-Index-Proposition}
Assume $F$ is a local field with $\overline{V}$ finite.
If $A_f$ is a maximal order, then the Schur index of $\Sigma_f$ is $$e(S\!\!\mid\!\! F)\cdot\left( \mid\!\! G\!\!\mid/ \mid\!\! H\!\!\mid\right).$$
\end{corollary}

\begin{proof}
Suppose $A_f$ is a maximal order in $\Sigma_f$.
 By Proposition~\ref{prop:Harada-result}, $K/F$ is tamely ramified. Let $n$ be the matrix size of $\Sigma_f$. If $s$ is the Schur index of $\Sigma_f$, then $\mid\!\! G\!\!\mid^2=n^2s^2$.

 By \cite[Theorem F]{Wa}, $n$ is also the matrix size of $\overline{A}_f$. But $\overline{A}_f=\overline{S(H)}$ and $S(H)$ is a maximal order in $K(H)$. Therefore the matrix size of $K(H)$ is $n$.

Let $e=e(S\!\!\mid\!\! F)$. By a result of Wilson, $G^T\subseteq H$ (see Proposition~\ref{prop:A-generalized-Wilson-Proposition}(1)). Therefore, if $L$ is the fixed field of $H$, then $e=e(S\!\!\mid\!\! L)$. Thus, if $U=S\cap L$, then $e$ is the smallest positive integer such that $J(S)^e=J(U)S$.

Let $\Delta$ be the invariant valuation ring of the division algebra part of $K(H)$ with $Z(\Delta)=U$. Since $J(S(H))=\sum_{\sigma\in H}J(S)x_{\sigma}$ and $S(H)\cong M_n(\Delta)$, $e$ is the smallest positive integer such that $J(\Delta)^e=J(U)\Delta$. By \cite[Theorem 14.3]{R}, the Schur index of $K(H)$ is $e$ as $\overline{V}$ is finite. Therefore $\mid\!\! H\!\!\mid^2=n^2e^2$, hence $s=e(S\!\!\mid\!\! F)\cdot\left( \mid\!\! G\!\!\mid/ \mid\!\! H\!\!\mid\right)$.
\end{proof}

If $F=\mathbb{Q}_p$ is the $p$-adic field and $H=G$, then the corollary above follows from \cite[Theorem 1]{J}. The point of the corollary above is that, in general, the Schur index is scaled up by the factor $\mid\!\! G\!\!\mid/ \mid\!\! H\!\!\mid$.

Given a maximal ideal $M$ of $S$, let $U=S_M\cap K^Z$
 and let $(K',S')$ be a Henselization of
$(K,S_M)$. Let $(F_h,V_h)$ be the unique Henselization of
$(F,V)$ contained in $(K',S')$. As before, we observe that $(F_h,V_h)$ is also a
Henselization of $(K^Z,U)$.

\begin{theorem} \label{thm:The-Main-Dubrovin-Theorem}
Suppose $K/F$ is tamely ramified and defectless.
 Then the crossed-product order $A_f$ is a valuation ring of $\Sigma_f$ if and only if, for some maximal ideal $M$ of $S$, $A_{f_M}$ is a valuation ring of $\Sigma_{f_M}$ and there exists a set of right coset representatives $\sigma_1,\sigma_2,\ldots,\sigma_r$ of $G^Z(S_M\!\!\mid\!\! F)$ in $G$ such that for all $i$, $f(\sigma_i,\sigma_i^{-1})\not\in M$.

 In particular if $H=G$ and $K/F$ is tamely ramified and defectless, then the crossed-product order $A_f$ is a valuation ring of $\Sigma_f$ if and only if, for some maximal ideal $M$ of $S$, $A_{f_M}$ is a valuation ring of $\Sigma_{f_M}$; if and only if, for some maximal ideal $M$ of $S$, $A_{f_{(M)}}$ is a valuation ring of $\Sigma_{f_{(M)}}$.
\end{theorem}

\begin{proof}
The result obviously holds when $V$ is indecomposed in $K$, so assume that $V$ decomposes in $K$.
Suppose $A_f$ is a valuation ring of $\Sigma_f$. Let $M$ be a maximal ideal of $S$. By Proposition~\ref{prop:A-global-local-result} and Lemma~\ref{lem:Primarity-Lemma}, $A_{f_M}$ is a valuation ring of $\Sigma_{f_M}$ and
there is a set of right coset representatives $\sigma_1,\sigma_2,\ldots,\sigma_r$ of $G^Z$ in $G$ such that $f(\sigma_i,\sigma_i^{-1})\not\in M$ for all $i$.

On the other hand, suppose there exists a set of right coset representatives $\sigma_1,\sigma_2,\ldots,\sigma_r$ of $G^Z$ in $G$ such that $f(\sigma_i,\sigma_i^{-1})\not\in M$ for all $i$. Then $A_f\otimes_VV_h\cong M_r(A_{f_M}\otimes_UV_h)$, as was seen in \cite[Lemma 2.7]{K6} (see also the remark following \cite[Theorem 3.12]{H}). Therefore, if in addition $A_{f_M}$ is a valuation ring of $\Sigma_{f_M}$, then $A_{f_M}\otimes_U V_h$ is a valuation ring of $\Sigma_{f_M}\otimes_{K^Z}F_h$ by \cite[Theorem F]{Wa}, hence $M_r(A_{f_M}\otimes_UV_h)$ is a valuation ring of $M_r(\Sigma_{f_M}\otimes_{K^Z}F_h)=\Sigma_f\otimes_FF_h$ by \cite[\S 1, Theorem 7]{D1}, and so $A_f$ is a valuation ring of $\Sigma_f$.
\end{proof}

An easy application of Theorem~\ref{thm:The-Main-Dubrovin-Theorem} can be made to the crossed-product order $A_{f_2}$ given in \cite[Example 4.5]{H}; if we set $f=f_2$ we see that given any maximal ideal $M$ of $S$ in that example,
there is indeed a set of right coset representatives of $G^Z$ in $G$ with the desired property. Further, $H_M=\{1\}$ and $f_M(\sigma,\tau)\not\in M^2$
for each $\sigma,\tau\in G^Z$. Hence, by Proposition~\ref{prop:Cyclic-Division-Algebra}, $A_{f_M}$ is a valuation ring of $\Sigma_{f_M}$ and so $A_f$ is indeed a valuation ring of $\Sigma_f$
as was already demonstrated in \cite{H}. As an aside, we make the following quick observation: in the example in \cite{H}, $f$ $(=f_2)$ was constructed from a two-cocycle $\overline{f}$
in $Z^2(G^Z,S_M^{\#})$ in such a manner that $\textrm{Gr}(\overline{f})=\textrm{Gr}(f_M)$,
using an intricate procedure called ``lifting'' which was developed in that paper. Here the reverse happened: we have been handed down an already constructed $f$ from which $f_M$ has been effortlessly extracted.

Said differently, Theorem~\ref{thm:The-Main-Dubrovin-Theorem} says that if $K/F$ is tamely ramified and defectless, then the crossed-product order $A_f$ is a valuation ring of $\Sigma_f$ if and only if it is primary and there exists a maximal ideal $M$ of $S$ such that $A_{f_M}$ is a valuation ring of $\Sigma_{f_M}$. When this happens, $A_{f_N}$ is a valuation ring of $\Sigma_{f_N}$ for every maximal ideal $N$ of $S$. When $H=G$, then the result above can also be found in \cite[Proposition 1]{K4}.

\begin{remark} \label{rem:Dubrovin-Theorem-remark}
Therefore, if $K/F$ is tamely ramified and defectless, the crossed-product order $A_f$ is a valuation ring of $\Sigma_f$ if and only if the following three conditions hold: it is semihereditary; for each maximal ideal $M$ of $S$, $A_{f_{(M)}}$ is a valuation ring of $\Sigma_{f_{(M)}}$;
and there exists a set of right coset representatives $\sigma_1,\sigma_2,\ldots,\sigma_r$ of $G^Z$ in $G$ such that for all $i$, $f(\sigma_i,\sigma_i^{-1})\not\in M$.  The existence of a ``nice'' set of representatives is necessary here (cf. Remark~\ref{rem:Kessler-Wilson}): if we do not have a ``nice'' set of coset representatives, then $A_{f_M}$ could be a valuation ring of $\Sigma_{f_M}$ for every maximal ideal $M$ of $S$ without $A_f$ being a valuation ring of $\Sigma_f$ (see \cite[Example]{K5}).

We have effective criteria for determining when $A_f$ is semihereditary. Assuming that we can get hold of ``nice'' coset representatives of $G^Z$ in $G$ as above, the program of determining when the crossed-product order $A_f$ is a valuation ring of $\Sigma_f$ really boils down to determining when $A_{f_{(M)}}$ is a valuation ring of $\Sigma_{f(M)}$.
It is therefore warranted to focus our attention on the case when $K/F$ is tame totally ramified due to the preceding results. By the condition given in Proposition~\ref{prop:Valuation-Ring-Proposition}(2)(c), we may even restrict ourselves to the case when $H=G$. Here is a special case of this instance.
\end{remark}

\begin{proposition} \label{prop:A-Generalized-Janusz-Result}
Suppose $K/F$ is tame totally ramified, $S$ is finitely generated over $V$, $H=G$, and $(F,V)$ is Henselian. Then the following are equivalent:
\begin{enumerate}
    \item The crossed-product order $A_f$ is a valuation ring of $\Sigma_f$.
    \item $\Sigma_f$ is a division algebra.
\end{enumerate}
When this happens, $A_f$ is an invariant valuation ring of $\Sigma_f$ and $\Sigma_f$ is a cyclic division algebra.
\end{proposition}

\begin{proof}
If $J(V)$ is not a principal ideal of $V$, then $e(S\!\!\mid\!\! F)=\epsilon(S\!\!\mid\!\! F)=1$, by \cite[Theorem 18.3 \& Theorem 18.6]{E}, hence $K=F$ and we are done. So assume $J(V)$ is a principal ideal of $V$.

Note that $G=G^T(S\!\!\mid\!\! F)$.  Since $K/F$ is tamely ramified and defectless, if $\overline{p}$ is the characteristic exponent of $\overline{V}$ then $\gcd(\mid\!\! G^T\!\!\mid,\overline{p})=1$ by \cite[Lemma 1(a)]{K3}. Therefore by \cite[20.18, 20.12, \& 20.2]{E},
 we have $G^T\cong\Gamma_S/\Gamma_V$, since $\Gamma_S/\Gamma_V$ is $\overline{p}$-free. The remark before Lemma~\ref{lem:A-generalized-Wilson-Theorem} shows that $G$ is a cyclic group, since $e(S\!\!\mid\!\! F)=\epsilon(S\!\!\mid\!\! F)$ as $S$ is finitely generated over $V$ by
\cite[Theorem 18.6]{E}. Let $G=\langle\sigma\rangle$, $n=\mid\!\! G\!\!\mid$, and let $a=\prod_{i=0}^{n-1}f(\sigma^i,\sigma)\in U(V)$.

Suppose $A_f$ is a valuation ring of $\Sigma_f$. Note that  $$\overline{A}_f\cong \overline{V}^t[G]\cong\overline{V}[X]/(X^n-\overline{a}),$$ a commutative ring. Since $A_f$ is primary, $\overline{A}_f$ must be a field. This implies that $\Sigma_f$ is a division algebra and $A_f$ is an invariant valuation ring of $\Sigma_f$, as we saw in the proof of Proposition~\ref{prop:Cyclic-Division-Algebra}.

Now suppose $\Sigma_f$ is a division algebra. Since $(F,V)$ is Henselian, $\Sigma_f$ admits an invariant valuation ring extending $V$. Since $H=G$ and $K/F$ is tamely ramified and defectless, $A_f$ is semihereditary, hence
it is the invariant valuation ring of $\Sigma_f$ lying over $V$ by \cite[Theorem 1.9]{K1}.
\end{proof}

Note that, in the proposition above, $\Sigma_f$ is a division algebra if and only if the ramification index of $S$ over $F$ is equal to the Schur index of $\Sigma_f$. In other words, the proposition above says that, with the rather strong assumptions given, $A_f$ is a valuation ring of $\Sigma_f$ if and only if the ramification index of $S$ over $F$ is equal to the Schur index of $\Sigma_f$ (cf. \cite[Theorem 1]{J}).

\section{The graphs of cocycles}
\label{sec:graphs}

Let $M$ be a maximal ideal of $S$. For a $\sigma\in G^Z$, the assignment $\sigma H_M\mapsto [\sigma]_M$ is a well-defined graph monomorphism $\psi_M:\textrm{Gr}(f_M)\mapsto \textrm{Gr}(f^M)$ which embeds $\textrm{Gr}(f_M)$ inside $\textrm{Gr}(f^M)$.
In particular, if $\textrm{Gr}(f^M)$ is a chain then so is $\textrm{Gr}(f_M)$. As in \cite[Proposition 3.4(c)]{H}, the canonical map $\textrm{Gr}(f)\mapsto\prod_{M\, \textrm{max}} \textrm{Gr}(f^M)$ is injective.
Also note that, when $V$ is indecomposed in $K$, then for each $\sigma\in G$ the equivalence class $[\sigma]_{J(S)}$ is a singleton, and the graphs $\textrm{Gr}(f)$, $\textrm{Gr}(f_{J(S)})$, and $\textrm{Gr}(f^{J(S)})$ all coincide.

 If there exists a set of right coset representatives $\sigma_1,\sigma_2,\ldots,\sigma_r$ of $G^Z$ in $G$  such that for all $i$, $f(\sigma_i,\sigma_i^{-1})\not\in M$ then the argument in \cite[Proposition 3.4(b)]{H}, and the remarks following \cite[Proposition 3.3]{H} show that the map $\phi_M:\textrm{Gr}(f)\mapsto \textrm{Gr}(f_M)$ given by $\phi_M(\sigma H)=dH_M$
where $d\in G^Z$ satisfies $f(d^{-1}\sigma,\sigma^{-1}d)\not\in M$ is a well-defined graph epimorphism. Let $c_M:\textrm{Gr}(f)\mapsto \textrm{Gr}(f^M):\sigma H\mapsto[\sigma]_M$ be the canonical graph epimorphism from $\textrm{Gr}(f)$ to $\textrm{Gr}(f^M)$. The connections between the graphs $\textrm{Gr}(f)$, $\textrm{Gr}(f_M)$ and $\textrm{Gr}(f^M)$ are summed up in the following theorem. The theorem also characterizes the ``nice'' set of coset representatives of $G^Z$ in $G$, which has been mentioned earlier, in terms of the graph homomorphism $\psi_M$.

\begin{theorem} \label{thm:The-Graph-of-Primary-Order}

Let $N$ be a maximal ideal of $S$.
\begin{enumerate}
    \item The graph homomorphism $\psi_N$ is an isomorphism if and only
    if there exists a set of right coset representatives $\sigma_1,\sigma_2,\ldots,\sigma_r$ of $G^Z(S_N\!\!\mid\!\! F)$ in $G$  such that for all $i$, $f(\sigma_i,\sigma_i^{-1})\not\in N$. When this happens, then $\phi_N$ is a well-defined graph epimorphism and
    the following diagram commutes:
        $$\xymatrix{\textrm{Gr}(f) \ar[rr]^{\phi_N}\ar[dr]_{c_N} && \textrm{Gr}(f_N) \ar[dl]^{\psi_N}\\
 & \textrm{Gr}(f^N) & }$$
    \item If $\psi_N$ is a graph isomorphism and $M$ is another maximal ideal of $S$ then:
     \begin{enumerate}
     \item the graphs $\textrm{Gr}(f^M)$, $\textrm{Gr}(f^N)$, $\textrm{Gr}(f_M)$ and $\textrm{Gr}(f_N)$ are all isomorphic (cf. \cite[Proposition 3.7(2c)]{H}), and
      \item $\psi_M$ is also a graph isomorphism.
      \end{enumerate}
 \end{enumerate}
 In particular, if $K/F$ is tamely ramified and defectless and the crossed-product order $A_f$ is primary then
 $\psi_N$ is a graph isomorphism; the diagram above is always commutative; and as $M$ ranges over all maximal ideals of $S$, the various graphs $\textrm{Gr}(f^M)$ and $\textrm{Gr}(f_M)$ are all isomorphic to one another.
\end{theorem}

\begin{proof} Part (1) follows from the discussion following \cite[Proposition 3.3]{H}. For suppose $\psi_N$ was an isomorphism. Let $\gamma_1,\gamma_2,\ldots,\gamma_r$ be right coset representatives of $G^Z(S _N\!\!\mid\!\! F)$ in $G$. Then there are elements $d_i\in G^Z(S_N\!\!\mid\!\! F)$ such that $[d_i]_N=[\gamma_i]_N$ for each $i$. Then the elements $\sigma_i=d_i^{-1}\gamma_i$ are right coset representatives of $G^Z(S_N\!\!\mid\!\! F)$ in $G$ such that for all $i$ we have $f(\sigma_i,\sigma_i^{-1})\not\in N$, since
$f^{d_i}(d_i^{-1}\gamma_i,\gamma_i^{-1}d_i)=f(d_i,d_i^{-1}\gamma_i)f(\gamma_i,\gamma_i^{-1}d_i)$. Conversely, suppose we have such a set of right coset representatives of $G^Z(S_N\!\!\mid\!\! F)$ in $G$. Let $\gamma\in G$. Then there is $d\in G^Z(S_N\!\!\mid\!\! F)$ such that $\gamma=d\sigma_i$ for some $i$. Clearly, $\psi_M(dH_N)=[\gamma]_N$ so that $\psi_N$ is an isomorphism.

For part (2),
let $M$ be a maximal ideal of $S$.
By part (1), there exists a set of right coset representatives $\sigma_1,\sigma_2,\ldots,\sigma_r$ of $G^Z(S_N\!\!\mid\!\! F)$ in $G$ such that for all $i$, $f(\sigma_i,\sigma_i^{-1})\not\in N$. Since the $\sigma_i^{-1}$ form a complete set of \textit{left} coset representatives of $G^Z(S_N\!\!\mid\!\! F)$ in $G$, there is a
$\sigma\in\{\sigma_1,\sigma_2,\ldots,\sigma_r\}$ such that $M^{\sigma}=N$. We claim that there is a graph isomorphism from $\textrm{Gr}(f^M)$ to $\textrm{Gr}(f^N)$ given by

 $$[\gamma]_M\mapsto [\sigma\gamma]_N.$$

For let $\gamma_1,\gamma_2\in G$. Then $f(\sigma,\gamma_i)\not\in N$, since $f^{\sigma}(\sigma^{-1},\sigma\gamma_i)f(\sigma,\gamma_i)=f(\sigma,\sigma^{-1})\not\in N$.
If $[\gamma_1]_M=[\gamma_2]_M$, then $f(\gamma_1,\gamma_1^{-1}\gamma_2), f(\gamma_2,\gamma_2^{-1}\gamma_1)\not\in M$, hence $$f^{\sigma}(\gamma_1,\gamma_1^{-1}\gamma_2), f^{\sigma}(\gamma_2,\gamma_2^{-1}\gamma_1)\not\in M^{\sigma}=N.$$ But $$f^{\sigma}(\gamma_1,\gamma_1^{-1}\gamma_2)f(\sigma,\gamma_2)=
f(\sigma,\gamma_1)f(\sigma\gamma_1,(\sigma\gamma_1)^{-1}(\sigma\gamma_2)),$$ and
$$f^{\sigma}(\gamma_2,\gamma_2^{-1}\gamma_1)f(\sigma,\gamma_1)=
f(\sigma,\gamma_2)f(\sigma\gamma_2,(\sigma\gamma_2)^{-1}(\sigma\gamma_1)).$$ Therefore
$[\sigma\gamma_1]_N=[\sigma\gamma_2]_N$ and the map above is indeed well-defined. Observe that the preceding argument is fully reversible: if $[\sigma\gamma_1]_N=[\sigma\gamma_2]_N$, then one can use the same steps in reverse to establish that $[\gamma_1]_M=[\gamma_2]_M$. Therefore the map is one-to-one as well. It is clearly onto.

 Now if $[\gamma_1]_M\leq_M[\gamma_2]_M$, then $f(\gamma_1,\gamma_1^{-1}\gamma_2)\not\in M$ hence $f^{\sigma}(\gamma_1,\gamma_1^{-1}\gamma_2)\not\in M^{\sigma}=N$.
But $f(\sigma,\gamma_2)\not\in N$ and $$f^{\sigma}(\gamma_1,\gamma_1^{-1}\gamma_2)f(\sigma,\gamma_2)=
f(\sigma,\gamma_1)f(\sigma\gamma_1,(\sigma\gamma_1)^{-1}(\sigma\gamma_2)).$$ Therefore $[\sigma\gamma_1]_N\leq_N[\sigma\gamma_2]_N$. Conversely, if $[\sigma\gamma_1]_N\leq_N[\sigma\gamma_2]_N$ then
$[\gamma_1]_M\leq_M[\gamma_2]_M$, as can be seen by simply reversing the preceding steps as before. Thus the graphs
 $\textrm{Gr}(f^M)$ and $\textrm{Gr}(f^N)$ are indeed isomorphic, and by part (1) they are both isomorphic to $\textrm{Gr}(f_N)$.

Now set $\tau_i=\sigma^{-1}\sigma_i$. Then $f(\tau_i,\tau_i^{-1})\not\in M$, since $f(\sigma_i,\sigma_i^{-1})\not\in N$. Further, $\tau_i\tau_j^{-1}(M)=M$ if and only if $\sigma_i\sigma_j^{-1}(N)=N$, so that
 $\{\tau_1,\tau_2,\ldots,\tau_r\}$ forms a complete set of right coset representatives of $G^Z(S_M\!\!\mid\!\! F)$ in $G$.
 Thus $\psi_M$ is also an isomorphism of graphs.
\end{proof}

From \cite[Example]{K5}, we see that the graphs $\textrm{Gr}(f_M)$ and $\textrm{Gr}(f^M)$ may not be isomorphic for any maximal ideal $M$ of $S$ in general. Also, we see from the same example that you could have the graphs $\textrm{Gr}(f^M)$ and $\textrm{Gr}(f^N)$  isomorphic for every choice of maximal ideals $M$ and $N$ of $S$ without $\psi_X$ being a graph isomorphism for any maximal ideal $X$ of $S$.

Further examples of graphs that would illustrate the various phenomena described in the theorem above can be found in \cite[\S 4]{H}. All the crossed-product orders in \cite[\S 4]{H} are primary, and the extension $K/F$ is tamely ramified and defectless.

For the rest of this section we shall make use of the following observation. Let $M$ be a maximal ideal of $S$ and assume there exists a set of right coset representatives $\sigma_1,\sigma_2,\ldots,\sigma_r$ of $G^Z$ in $G$ such that for all $i$, $f(\sigma_i,\sigma_i^{-1})\not\in M$. Then, as we saw in the previous section, the following holds: $\overline{A}_f\cong M_r(\overline{A}_{f_M})$ and, if $(F_h,V_h)$ is the Henselization of
$(F,V)$, we have $A_f\otimes_VV_h\cong M_r(A_{f_M}\otimes_UV_h)$ where $U=S_M\cap K^Z$. Note that all this is independent of whether or not the extension $K/F$ is tamely ramified or defectless.

\begin{corollary} \label{cor:The-Graph-Thm-Cor1}
Suppose there exists a maximal ideal $M$ of $S$ such
that $\psi_M$ is a graph isomorphism.
Then the crossed-product order $A_f$ is a valuation ring
of $\Sigma_f$ (respectively is primary) if and only if
the crossed-product order $A_{f_M}$ is a valuation ring of $\Sigma_{f_M}$ (respectively is primary).
\end{corollary}

\begin{remark} \label{rem:maximal-ideal-remark}
In the corollary above as well as elsewhere in this section, when a stated condition holds
with respect to one maximal ideal of $S$, then the same condition holds for every maximal ideal of $S$.
This is due to Theorem~\ref{thm:The-Graph-of-Primary-Order}.
\end{remark}

\begin{corollary} \label{cor:The-Graph-Thm-Cor2}
Suppose there exists a maximal ideal $M$ of $S$ such
that $\psi_M$ is a graph isomorphism. Then the following hold:
\begin{enumerate}
    \item There is a one-to-one correspondence between the ideals of $A_f$ containing $J(A_f)$ and
    the ideals of $A_{f_M}$ containing $J(A_{f_M})$. This correspondence between the ideals preserves
    inclusion, products, and intersections. In particular, the number of maximal ideals of $A_f$ and $A_{f_M}$
    is the same (cf. \cite[Theorem 3.10]{H}).
    \item The crossed-product order $A_f$ is semihereditary if and only if the crossed-product order $A_{f_M}$ is semihereditary.
    \end{enumerate}
\end{corollary}

\begin{proof} Part (1) follows from the fact that $\overline{A}_f\cong M_r(\overline{A}_{f_M})$. Since
$A_f\otimes_VV_h\cong M_r(A_{f_M}\otimes_UF_h)$, part (2) follows from \cite[Theorem 3.4]{K1}.\end{proof}

Part (2) of the corollary above generalizes \cite[Theorem 4.14]{CW2}. Also note that
if $H=G$, then $\psi_M$ is always a graph isomorphism, hence the conclusions of the corollary above always hold. But in general, if $\psi_M$ is not a graph isomorphism, then the conclusions of the corollary above my not hold (see \cite[Example]{K5}).

\section*{Acknowledgment}
\label{sec:ack}

The author wishes to thank the referee for his/her detailed and helpful comments which have
resulted in an improved exposition.



%% References


\bibliographystyle{model1-num-names}
\bibliography{<your-bib-database>}

\begin{thebibliography}{00}



\bibitem {AR} M. Auslander, D. S. Rim, Ramification index and multiplicity, Illinois J. Math. 7(1963)
566--581.

\bibitem {D1} N. I. Dubrovin, Noncommutative valuation
rings, Trudy Moskov. Mat. Obschch. 45(1982) 265-289;
English translation, Trans. Moskow Math. Soc. 45(1984)
273-287.

\bibitem {E} O. Endler, Valuation Theory, Springer-Verlag, New York, 1972.

\bibitem {H} D. E. Haile, Crossed-Product Orders over Discrete
Valuation Rings, J. Algebra 105(1987) 116--148.

\bibitem {HMW} D. E. Haile, P. J. Morandi, A. R. Wadsworth,
B\'{e}zout orders and Henselization, J. Algebra
173(1995) 394--423.

\bibitem {Hd} M. Harada, Some criteria for hereditarity of crossed products, Osaka J. Math. 1(1964)
69--80.

\bibitem {J} G. J. Janusz, Crossed product orders and the Schur index,
Comm. Algebra 8(7)(1980) 697--706.

\bibitem {K1} J. S. Kauta, Integral Semihereditary Orders, Extremality, and Henselization,
J. Algebra 189(1997) 226--252.

\bibitem {K2} J. S. Kauta, On Semihereditary Maximal Orders,
Bull. London Math. Soc. 30(1998) 251--257.

\bibitem {K3} J. S. Kauta, Crossed product orders over valuation rings,
Bull. London Math. Soc. 33(2001) 520--526.

\bibitem {K4} J. S. Kauta, H. Marubayashi, H. Miyamoto, Crossed product orders over valuation rings II:
tamely ramified crossed product algebras,
Bull. London Math. Soc. 35(2003) 541--552.

\bibitem {K5} J. S. Kauta, On a class of hereditary crossed-product orders,
 Proc. Amer. Math. Soc., 141(5)(2013) 1545 -- 1549.

\bibitem {K6} J. S. Kauta, On a class of semihereditary crossed-product orders,
Pacific J. Math. 259(2)(2012) 349--360.

\bibitem {Ke} V. Kessler, Crossed product orders and non-commutative arithmetic,
J. Number Theory 46(1994) 255--302.

\bibitem {MMUZ} H. Marubayashi, H. Miyamoto, A. Ueda, Y. Zhao,
Semi-hereditary orders in a simple Artinian ring,
Comm. Algebra 22(1994) 5209--5230.

\bibitem {MR} J. C. McConnell, J. C. Robson, Noncommutative Noetherian Rings, in `Graduate Studies in Math.,' Vol. 30, American Math. Soc., Providence, RI, 1987.

\bibitem {M} P. J. Morandi, Maximal Orders over Valuation Rings, J. Algebra 152(1992) 313--341.

\bibitem {M2} P. J. Morandi, Noncommutative Pr\"{u}fer rings satisfying a polynomial identity, J. Algebra 161(1993), 324--341.

\bibitem {P} D. S. Passman, Infinite Crossed Products, Academic Press, San Diego, 1989.

\bibitem {R} I. Reiner, Maximal Orders, Academic Press, London, 1975.

\bibitem {Rt} J. J. Rotman, An Introduction to Homological Algebra, 2d ed., in `Universitext', Springer, New York, 2009.

\bibitem {Wa} A. R. Wadsworth, Dubrovin
valuation rings and Henselization, Math. Ann. 283(1989)
301--328.

\bibitem {W} S. Williamson, Crossed products and hereditary orders,
Nagoya Math. J. 23(1963) 103--120.

\bibitem{CW1} C. J. Wilson, Hereditary crossed product orders over discrete valuation rings, Ph.D. Thesis, Indiana University, 2011.

\bibitem{CW2} C. J. Wilson, Hereditary crossed product orders over discrete valuation rings, J. Algebra 371(2012) 329-349.

\bibitem {Y} Z. Yi, Homological Dimension of skew group rings and crossed products, J. Algebra 164(1994) 101--123.
\end{thebibliography}



\end{document}